\documentclass[french,11pt,a4paper,leqno]{amsart}
\usepackage{amsmath,amstext, amsthm, amssymb}
\usepackage{latexsym, amsfonts, graphicx, xcolor}
\usepackage[latin1]{inputenc} 
\usepackage[T1]{fontenc} 
\usepackage{enumerate}
\usepackage[french]{babel} 
\usepackage{enumitem}
\usepackage{hyperref}

\makeatletter
\def\namedlabel#1#2{\begingroup
    #2%
    \def\@currentlabel{#2}%
    \phantomsection\label{#1}\endgroup
}

\makeatother

\hypersetup{
  colorlinks   = true, 
  urlcolor     = blue, 
  linkcolor    = blue, 
  citecolor   = red 
}

\numberwithin{equation}{section} 

\newtheorem{thm}{Th\'eor\`eme}[section]
\newtheorem{prop}[thm]{Proposition}
\newtheorem{lem}[thm]{Lemme}
\newtheorem{coro}[thm]{Corollaire}

\newtheorem*{prop*}{Proposition}
\newtheorem*{thmE1}{Th\'eor\`eme E1}
\newtheorem*{thmE2}{Th\'eor\`eme E2}
\newtheorem*{thmM1}{Th\'eor\`eme M1}
\newtheorem*{thmM2}{Th\'eor\`eme M2}

\theoremstyle{remark}
\newtheorem{rem}[thm]{Remarque}

\theoremstyle{definition}
\newtheorem{defi}[thm]{D\'efinition}
\newtheorem{ex}[thm]{Exemple}

\newtheorem*{nota*}{Notation}

\newcommand{\Q}{\overline{\mathbb Q}}

\newcommand{\N}{\mathbb{N}}

\newcommand{\X}{\boldsymbol{X}}

\title[Valeurs de $E$-fonctions ou de $M$-fonctions ]{
Relations alg\'ebriques entre valeurs de $E$-fonctions ou de $M$-fonctions }

\author{Boris Adamczewski}
\address{
Univ Lyon, Universit\'e Claude Bernard Lyon 1\\
 CNRS UMR 5208, Institut Camille Jordan \\
 F-69622 Villeurbanne Cedex, France}
\email{Boris.Adamczewski@math.cnrs.fr}

\author{Colin Faverjon}
\address{
Univ Lyon, Universit\'e Claude Bernard Lyon 1\\
 CNRS UMR 5208, Institut Camille Jordan \\
 F-69622 Villeurbanne Cedex, France}
\email{faverjon@math.univ-lyon1.fr}
\date{}

\thanks{ }
\begin{document}

\begin{abstract}
Nous montrons que \emph{toutes} les relations algébriques sur $\Q$ entre les valeurs prises par des $E$-fonctions de Siegel en un point algébrique non nul sont d'origine fonctionnelle, en ce sens qu'elles s'obtiennent par dégénérescence de relations algébro-différentielles sur $\Q(z)$ entre les fonctions considérées. 
Nous obtenons un résultat analogue pour les $M_q$-fonctions de Mahler, dans lequel les relations dites $\sigma_q$-alg\'ebriques  se substituent aux relations alg\'ebro-diff\'erentielles. 
Nous donnons également plusieurs conséquences de ce résultat, notamment concernant certains phénomènes de descente.  
Le point de vue adopté révèle 
des similitudes frappantes entre la théorie des $E$-fonctions et celle des $M_q$-fonctions.  
\end{abstract}
\bibliographystyle{abbvr}
\maketitle

\section{Introduction}

La th\'eorie des nombres transcendants est mue par deux objectifs principaux. Un ensemble de nombres complexes $\Xi$ \'etant fix\'e, il s'agit d'une part d'\^etre capable de d\'eterminer, pour tous  
$\xi_1,\ldots,\xi_r\in\Xi$, l'ensemble des relations alg\'ebriques ou lin\'eaires entre ces nombres sur le corps $\Q$ des nombres alg\'ebriques ou un sous-corps de ce dernier et, d'autre part, d'en trouver la \emph{raison d'\^etre}. Ces deux probl\`emes sont naturellement  li\'es et le second, plus ambigu, d\'epend bien s\^ur de la fa\c con dont les \'el\'ements de $\Xi$ sont d\'efinis. 

Par exemple,  
si $\Xi$ d\'esigne l'anneau des p\'eriodes, une conjecture de Grothendieck pr\'edit que les relations alg\'ebriques sont n\'ecessairement d'origine motivique, ce qui r\'epondrait dans ce cadre au second objectif tout en offrant un outil puissant pour atteindre le premier. Une formulation plus \'el\'ementaire, bien qu'\'equivalente, est donn\'ee par la conjecture de Kontsevich et Zagier qui 
stipule que toute relation alg\'ebrique entre p\'eriodes d\'ecoule des r\`egles fondamentales de l'int\'egration que sont l'additivité, le changement de variables et la formule de Stokes.   
Pour davantage de détails sur ces deux conjectures et leurs liens voir \cite{Fre22,KZ}. 
Si à présent on choisit pour $\Xi$ l'ensemble $\{(1/\sqrt{2\pi})\Gamma(r) : r\in\mathbb Q\}$, où $\Gamma$ est la fonction gamma d'Euler, la conjecture de Rohrlich-Lang pr\'edit que les relations alg\'ebriques proviennent n\'ecessairement de relations fonctionnelles strandard associées à $\Gamma$ (voir, par exemple, \cite[Conjecture 22]{Miw}). 
Outre le fait de s'inscrire dans une ambition commune, ces diff\'erentes conjectures ont pour point commun d'être consid\'er\'ees 
comme totalement hors de port\'ee des méthodes actuelles.

Consid\'erons des fonctions analytiques $f_1(z),\ldots,f_r(z)$ dont le d\'eveloppement de Taylor \`a l'origine est \`a coefficients alg\'ebriques et un nombre alg\'ebrique non nul $\alpha$ appartenant \`a un domaine d'analyticité commun contenant $0$. Dans ce contexte, les relations fonctionnelles, lorsqu'il en existe, sont une source évidente de relations entre les nombres $f_1(\alpha),\ldots,f_r(\alpha)$. En effet, s'il existe une relation alg\'ebrique homog\`ene sur $\Q(z)$ entre $f_1(z),\ldots,f_r(z)$, c'est-\`a-dire s'il existe un polyn\^ome $Q\in\Q[z,X_1,\ldots,X_r]$, homog\`ene en les variables $X_1,\ldots,X_r$, tel que 
$$
Q(z,f_1(z),\ldots,f_r(z))=0 \,,
$$
on obtient par sp\'ecialisation au point $\alpha$ une relation homog\`ene de m\^eme degr\'e, \`a savoir
\begin{equation*}\label{eq: relb}
P(f_1(\alpha),\ldots,f_r(\alpha))=0\,,
\end{equation*}
o\`u $P(X_1,\ldots,X_r):=Q(\alpha,X_1,\ldots,X_r)$.
Une telle relation sera dite \emph{banale}, tandis que les relations alg\'ebriques  ne pouvant s'obtenir de cette fa\c con seront qualifiées de \emph{non banales}.  
Trouver les relation banales entre les nombres $f_1(\alpha),\ldots,f_r(\alpha)$ revient donc \`a d\'eterminer l'idéal des relations alg\'ebriques sur $\Q(z)$ entre les fonctions $f_1(z),\ldots,f_r(z)$. 
Ce n'est g\'en\'eralement pas une t\^ache ais\'ee, mais lorsque les fonctions en question sont li\'ees par des \'equations diff\'erentielles ou aux diff\'erences lin\'eaires, 
les th\'eories galoisiennes associ\'ees  permettent d'obtenir des r\'esultats probants (voir \cite{vdPS2,vdPS1}).  

Notons que m\^eme dans des cas tr\`es simples, les relations banales n'\'epuisent pas n\'ecessairement l'ensemble des relations alg\'ebriques.  
Par exemple,  la fonction $f(z):=(z-1)e^z$ \'etant non nulle, la relation $f(1)=0$ ne peut \^etre banale.  On peut n\'eanmoins trouver une origine fonctionnelle \`a cette relation. En effet, 
celle-ci s'obtient par sp\'ecialisation au point $z=1$ de la relation diff\'erentielle 
\begin{equation}\label{eq: reld}
zf(z)-(z-1)f'(z)=0\,.
\end{equation}
Comme le coefficient de $f'$ s'annule en $z=1$, on dira que cette relation \emph{d\'eg\'en\`ere} au point $1$.  Le lecteur, ou la lectrice, prendra garde au fait que le caract\`ere banale 
d'une relation est relatif et se d\'etermine au regard d'un ensemble de fonctions pr\'ealablement fix\'e. Ainsi, la relation $f(1)=0$ est non banale relativement à l'ensemble $\{f(z)\}$, mais elle est banale,  
d'apr\`es \eqref{eq: reld}, relativement \`a l'ensemble $\{f(z),f'(z)\}$.  

Dans cet article, nous  étudions le cas où l'ensemble $\Xi$ est compos\'e des valeurs prises en un point $\alpha \in\Q^*$ 
par les éléments de l'anneau des $E$-fonctions de Siegel ou de celui des  $M_q$-fonctions de Mahler. 
Le th\'eor\`eme \ref{thm: orbital} ci-après montre que dans le cadre des $E$-fonctions, dont $f$ est justement un exemple, toutes les relations alg\'ebriques non banales s'obtiennent par dégénérescence de relations fonctionnelles alg\'ebro-diff\'erentielles, c'est-à-dire de fa\c con similaire à \eqref{eq: reld}. Il fournit \'egalement un r\'esultat analogue pour les $M_q$-fonctions, dans lequel les relations dites $\sigma_q$-alg\'ebriques  se substituent aux relations alg\'ebro-diff\'erentielles. 
Le point de vue adopté révèle 
des similitudes frappantes entre ces deux th\'eories qu'il serait intéressant de développer plus avant. Celles-ci sont d'autant plus étonnantes  que les $E$- et les $M_q$-fonctions sont de nature très différente. Par exemple, les $E$-fonctions sont entières, alors qu'une $M_q$-fonction qui n'est pas rationnelle  admet le cercle unité comme coupure et est  différentiellement transcendante (voir \cite{ADH,BCR}).

\subsection{R\'esultats principaux} 

Notons $\delta:=\frac{d}{dz}$ et, pour tout entier $k\geq 0$, $\delta^k(f)=f^{(k)}(z)$.  
Une $E$-fonction est une s\'erie formelle de la forme $f(z)=\sum_{n=0}^{\infty}\frac{a_n}{n!}z^n\in \Q[[z]]$
qui  satisfait \`a une \'equation diff\'erentielle lin\'eaire \`a 
coefficients dans $\Q[z]$,  c'est-\`a-dire qu'il existe des polyn\^omes $p_0(z),\ldots,p_m(z)\in\Q[z]$, non tous nuls, tels que 
$$
p_0f+p_1\delta(f)+\cdots +p_m\delta^{(m)}(f)=0 \,,
$$
et dont la croissance arithm\'etique des coefficients est limit\'ee par les deux conditions suivantes :  il existe $C>0$ et une suite d'entiers $d_n\geq 1$ tels que 
pour tout $\sigma\in \mbox{Gal}(\Q/\mathbb Q)$ et tout entier $n\geq 1$, on a $\vert \sigma(a_n)\vert \leq C^n$, $d_n\leq C^n$ et $d_na_i$ est un entier alg\'ebrique pour tout $i$, $1\leq i\leq n$. 

Pour tout entier $q\geq 2$, d\'esignons par $\sigma_q$ l'endomorphisme injectif de $\Q[[z]]$ d\'efini par $\sigma_q(f)=f(z^q)$, de sorte que, pour tout entier $k\geq 0$, 
$\sigma_q^k(f)=f(z^{q^k})$.  
Une $M_q$-fonction est une s\'erie formelle de la forme $f(z)=\sum_{n=0}^{\infty}a_nz^n\in \Q[[z]]$ qui satisfait \`a une équation aux différences $\sigma_q$-lin\'eaire \`a coefficients dans $\Q[z]$, 
c'est-\`a-dire qu'il existe des polyn\^omes $p_0(z),\ldots,p_m(z)\in\Q[z]$, non tous nuls, tels que 
\begin{equation}\label{eq: ma}
p_0f+p_1\sigma_q(f)+\cdots +p_m\sigma_q^m(f)=0 \,.
\end{equation}
Rappelons qu'une $M_q$-fonction est convergente (cf. \cite[Lemma 4]{BCR}) et admet donc, d'après \eqref{eq: ma} un prolongement méromorphe dans le disque unité ouvert.

\'Etant donn\'ees $f_1(z),\ldots,f_r(z)\in\Q[[z]]$, on notera 
$$
\mathfrak I(f_1,\ldots,f_r) := \left\{ Q\in \Q(z)[X_1,\ldots,X_r] : Q(z,f_1(z),\ldots,f_r(z))=0\right\}
$$
l'id\'eal des relations alg\'ebriques sur $\Q(z)$ entre ces s\'eries formelles. 
 L'id\'eal des relations algébro-différentielles, ou $\delta$-alg\'ebriques,  
 c'est-\`a-dire l'id\'eal des relations alg\'ebriques sur $\Q(z)$ entre 
$f_1(z),\ldots,f_r(z)$ et leurs d\'eriv\'ees successives, est not\'e 
$$
\mathfrak I^{\delta}(f_1,\ldots,f_r) :=\left\{ Q \in \Q(z)[(X_{i,j})_{1\leq i \leq r,\, j \in \N}] : Q(z,\delta^j(f_i(z))_{1\leq i \leq r,\, j \geq 0})=0\right\} \,.
$$
De fa\c con similaire, on note  $\mathfrak I^{\sigma_q}(f_1,\ldots,f_r)$ l'id\'eal des relations $\sigma_q$~-alg\'ebriques, c'est-\`a-dire l'id\'eal des relations alg\'ebriques sur $\Q(z)$ entre  
$f_1,\ldots,f_r$ et leurs images successives par l'op\'erateur $\sigma_q$. 
 Une relation $Q$, $\delta$-alg\'ebrique ou $\sigma_q$-alg\'ebrique,  est dite homog\`ene si $Q$ est homog\`ene en les variables $X_{i,j}$. 
En identifiant les variables $X_1,\ldots,X_r$ et $X_{1,0},\ldots,X_{r,0}$, on obtient que 
$$\mathfrak I(f_1,\ldots,f_r) \subset \mathfrak I^{\delta}(f_1,\ldots,f_r) \;\; \mbox{ et }\;\;  
\mathfrak I(f_1,\ldots,f_r) \subset \mathfrak I^{\sigma_q}(f_1,\ldots,f_r)\,.$$ 
Une relation $\delta$-algébrique  $Q\in Q[z,(X_{i,j})_{1\leq i \leq r,\, j \in \N}] $ \emph{dégénère} en $\alpha$ si 
\begin{equation*}
Q\in \mathfrak I^{\delta}(f_1,\ldots,f_r)\setminus\mathfrak I(f_1,\ldots,f_r) \;\;\mbox{ et }\;\;Q(\alpha,(X_{i,j})) \in \Q[X_1,\ldots,X_r]\,.
\end{equation*}
La dégénérescence d'une relation $\sigma_q$-algébrique est définie de façon similaire.

Nous pouvons \`a pr\'esent \'enoncer le r\'esultat principal de cet article.

\begin{thm}\label{thm: orbital} On a les deux r\'esultats suivants. 
	
	\begin{itemize}
\item[{\rm (E)}] Soient $f_1(z),\ldots,f_r(z)$ des $E$-fonctions et $\alpha \in \Q^*$. Toute relation alg\'ebrique homog\`ene sur $\Q$ non banale entre 
$f_1(\alpha),\ldots,f_r(\alpha)$ peut s'obtenir par dégénérescence en $\alpha$ d'une relation $\delta$-alg\'ebrique homog\`ene  entre les fonctions $f_1(z),\ldots,f_r(z)$. 

\item[{\rm(${\rm M}_q$)}] Soient $f_1(z),\ldots,f_r(z)$ des $M_q$-fonctions et $\alpha \in \Q$, $0<\vert\alpha\vert<1$, un point qui n'est un pôle d'aucune de ces fonctions. 
Toute relation alg\'ebrique homog\`ene sur $\Q$ non banale entre 
$f_1(\alpha),\ldots,f_r(\alpha)$ peut s'obtenir par dégénérescence en $\alpha$ d'une relation $\sigma_q$-alg\'ebrique homog\`ene  entre les fonctions $f_1(z),\ldots,f_r(z)$. 
 \end{itemize}
 \end{thm}
 
 \begin{rem}
 
 Dans le cas (E), ce r\'esultat  repose sur deux ingr\'edients  : un raffinement qualitatif du th\'eor\`eme de Siegel-Shidlovskii obtenu par Beukers \cite{Be06} et le fait, d\'emontr\'e par Andr\'e 
 \cite{An1}, que toute $E$-fonction est annul\'ee par un $E$-op\'erateur diff\'erentiel. Les th\'eor\`emes de Siegel-Shidlovskii  et de Beukers imposent aux fonctions $f_1(z),\ldots,f_r(z)$ d'\^etre li\'ees par un syst\`eme diff\'erentiel lin\'eaire et au point $\alpha$ d'\^etre un point r\'egulier pour ce syst\`eme (voir section \ref{sec: E}). L'apport du th\'eor\`eme \ref{thm: orbital} est de supprimer ces deux restrictions. La d\'emonstration du cas {\rm(${\rm M}_q$)} est similaire :  nous substituons au th\'eor\`eme de Beukers un r\'esultat analogue dû \`a Philippon \cite{PPH} et aux auteurs \cite{AF17}, et  nous introduisons la notion de $M_q$-op\'erateur qui remplace celle de $E$-op\'erateur.  
 
  La \emph{profondeur} d'une relation $Q$, $\delta$-alg\'ebrique ou $\sigma_q$-alg\'ebrique, 
  est le plus petit entier $s$ tel que le support de $Q$ est inclus dans $\{(i,j) : 1\leq i\leq r, 0\leq j\leq s\}$. 
 Dans le cas (E), nous verrons que toutes les relations algébriques s'obtiennent par dégénérescence de relations $\delta$-algébriques dont la profondeur est bornée indépendamment du point $\alpha$. 
 Ceci n'est plus vrai dans le cas {\rm(${\rm M}_q$)} (cf. exemple \ref{ex: tm3}), mais la profondeur peut toutefois être bornée en fonction de $R$ pour tout $\alpha\in\Q$ tel que $0< \vert \alpha\vert<R<1$. 
 \end{rem}

Soit $\alpha\in \Q^*$. Notons $\mbox{\bf E}\subset \mathbb C$ l'ensemble form\'e des \'evaluations de toutes les $E$-fonctions au point $\alpha$. 
Comme l'anneau des $E$-fonctions est stable par le changement de variabe $z\mapsto \alpha z$, $\mbox{\bf E}$ ne d\'epend pas du choix de $\alpha$. 
\'Etant donn\'e un sous-corps $\mathbb K$ de $\Q$, on note $\mbox{\bf E}_{\mathbb K}$ l'ensemble form\'e de l'\'evaluation en $z=1$ de toutes les $E$-fonctions \`a coefficients dans 
$\mathbb K$, de sorte que $\mbox{\bf E}=\cup_{\mathbb K}\mbox{\bf E}_{\mathbb K}$. De fa\c con similaire, notons $\mbox{\bf M}_{q,\alpha}$ l'ensemble form\'e des \'evaluations au point $\alpha$ de toutes les $M_q$-fonctions qui n'ont pas de pôle en $\alpha$ et 
$\mbox{\bf M}_{q,\alpha,\mathbb K}$  l'ensemble form\'e des \'evaluations au point $\alpha$ de toutes les $M_q$-fonctions \`a coefficients dans $\mathbb K$ qui n'ont pas de pôle en $\alpha$.  
Cette fois-ci, les ensembles obtenus d\'ependent du point $\alpha$. 
Les ensembles $\mbox{\bf E}$, $\mbox{\bf E}_{\mathbb K}$, $\mbox{\bf M}_{q,\alpha}$ et  $\mbox{\bf M}_{q,\alpha,\mathbb K}$ 
sont tous des anneaux.

Le r\'esultat suivant d\'ecoule directement du  th\'eor\`eme \ref{thm: orbital} (cf. section \ref{sec: descente}). 

\begin{coro}\label{coro: descente}
	Soient ${\mathbb K} \subset \Q$ un corps et $\alpha\in \mathbb K$, $0<\vert \alpha\vert<1$. On a les deux r\'esultats suivants. 
	
	\begin{itemize}
	
	\item[{\rm (E)}] Des \'el\'ements de $\mbox{\bf E}_{\mathbb K}$ sont lin\'eairement d\'ependants sur  	$\Q$ si, et seulement si, ils le sont sur $\mathbb K$. En d'autres termes, 
	les $\mathbb K$-alg\`ebres $\mbox{\bf E}_{\mathbb K}$ et $\Q$ sont lin\'eairement disjointes et $\mbox{\bf E}=\mbox{\bf E}_{\mathbb K}\otimes_{\mathbb K}\Q$. 
	
	\item[{\rm(${\rm M}_q$)}] Des \'el\'ements de $\mbox{\bf M}_{q,\alpha,\mathbb K}$ sont lin\'eairement d\'ependants sur  $\Q$ si, et seulement si, ils le sont sur $\mathbb K$. En d'autres termes, 
	les $\mathbb K$-alg\`ebres $\mbox{\bf M}_{q,\alpha,\mathbb K}$ et $\Q$ sont lin\'eairement disjointes et $$\mbox{\bf M}_{q,\alpha}=\mbox{\bf M}_{q,\alpha,\mathbb K}\otimes_{\mathbb K}\Q\,.$$ 
	
	\end{itemize}
\end{coro}

Le cas (E) a \'et\'e obtenu ind\'ependamment par Fischler et Rivoal \cite[Theorem 2]{FR23}. Le cas {\rm(${\rm M}_q$)} est dû aux auteurs \cite[th\'eor\`eme 1.7] {AF17}.  
Le th\'eor\`eme \ref{thm: orbital} 
permet d'obtenir une d\'emonstration unifi\'ee des deux cas. 

\subsection{Organisation de l'article}

La démonstration du théorème \ref{thm: orbital} occupe la section \ref{sec: thm}. Dans la section \ref{sec: eff}, nous expliquons brièvement comment le théorème \ref{thm: orbital} 
peut être utilisé pour atteindre le premier objectif évoqué, à savoir la détermination effective des relations algébriques entre les valeurs prises par des $E$- ou des 
$M_q$-fonctions fixées en un point algébrique $\alpha$ lui-aussi fixé.  
Une caractéristique fondamentale des relations banales est leur \emph{permanence} :  l'existence d'une relation banale en un point $\alpha$ implique l'existence d'une relation de m\^eme type 
(i.e. banale et homog\`ene de m\^eme degr\'e) en tout point alg\'ebrique non nul du domaine d'analyticité commun aux fonctions sous-jacentes. 
Dans le cadre des $E$- et des $M_q$-fonctions, nous montrons dans la section \ref{sec: spor} que les relations non banales sont quant à elles \emph{sporadiques}. 
Nous décrivons brièvement  la structure des idéaux de la forme $\mathfrak I^{\delta}(f_1,\ldots,f_r)$ et $\mathfrak I^{\sigma_q}(f_1,\ldots,f_r)$ dans la section \ref{sec: ideal}. 
Dans la section \ref{sec: descente}, nous montrons comment le corollaire \ref{coro: descente} découle du théorème \ref{thm: orbital} et nous mentionnons également quelques conséquences directes de ce résultat. La section \ref{sec: FR} est inspirée par l'article récent \cite{FR23} de Fischler et Rivoal. Nous y montrons comment utiliser le  théorème \ref{thm: orbital} pour obtenir l'analogue dans le cas des $M_q$-fonctions de certains de leurs résultats. Là encore, le théorème \ref{thm: orbital} permet de déduire de façon unifiée les résultats obtenus pour les $E$- et pour les $M_q$-fonctions. Dans la section \ref{sec: ex}, nous illustrons le théorème \ref{thm: orbital} à travers quelques exemples.

\section{Preuve du th\'eor\`eme \ref{thm: orbital} }\label{sec: thm}

Dans cette section, nous d\'emontrons le th\'eor\`eme \ref{thm: orbital}. 
Pour ce faire, nous rappelons quelques r\'esultats concernant les $E$- et les $M_q$-fonctions,  
puis nous introduisons la notion de $M_q$-op\'erateur (de niveau $R$) qui jouera dans la preuve du cas {\rm(${\rm M}_q$)} du th\'eor\`eme \ref{thm: orbital}, le r\^ole jou\'e par les 
$E$-op\'erateurs diff\'erentiels dans le cas (E).

\subsection{Rappels concernant la th\'eorie les $E$-fonctions}\label{sec: E}

\'Etant donn\'e un syst\`eme diff\'erentiel lin\'eaire d'ordre $1$ :
\begin{equation}\label{eq:DiffSystem}
{\bf Y}'(z)
=
A(z){\bf Y}(z)\,
,\quad\quad A(z) \in {\mathcal M}_m(\Q(z))\, ,
\end{equation} 
le point $\alpha\in\mathbb C$ est dit r\'egulier si la matrice $A(z)$ est bien d\'efinie en $\alpha$ et singulier sinon. 
Les singularit\'es d'un op\'erateur diff\'erentiel 
$$
L= a_0(z)+a_1(z)\delta+\cdots+a_m(z)\delta^{m} \in\Q[z,\delta] 
$$
sont les racines du polyn\^ome $a_m(z)$, lesquelles correspondent \'egalement aux singularités du syst\`eme diff\'erentiel associ\'e \`a la matrice compagnon de $L$ : 
$$
 A_L :=\begin{pmatrix}
0&1&0&\cdots&0\\
0&0&1&\ddots&\vdots\\
\vdots&\vdots&\ddots&\ddots&0\\
0&0&\cdots&0&1\\
-\frac{a_{0}(z)}{a_m(z)}& -\frac{a_{1}(z)}{a_m(z)}&\cdots & \cdots & -\frac{a_{m-1}(z)}{a_m(z)}
\end{pmatrix}  \,.
$$

Le r\'esultat fondamental de la th\'eorie des $E$-fonctions est le th\'eor\`eme de Siegel-Shidlovskii (voir, par exemple, \cite{Sh_Liv}). 

\begin{thmE1}\label{thm: trdegE}
	Soient $f_1(z),\ldots,f_m(z)$ des $E$-fonctions formant un vecteur solution d'un syst\`eme diff\'erentiel lin\'eaire de la forme \eqref{eq:DiffSystem} et 
	$\alpha \in \Q^*$ un point r\'egulier relativement \`a ce système. Alors, on a 
	$$
	{\rm deg.tr}_{\Q}(f_1(\alpha),\ldots,f_m(\alpha))={\rm deg.tr}_{\Q(z)}(f_1(z),\ldots,f_m(z))\,.
	$$
\end{thmE1}

Le th\'eor\`eme E1 est un \'enonc\'e quantitatif exprimant une \'egalit\'e de dimension, \`a savoir celle des dimensions de Krull des  anneaux 
$\Q[f_1(\alpha),\ldots,f_m(\alpha)]$ et $\Q(z)[f_1(z),\ldots,f_m(z)]$. Dans \cite{Be06}, Beukers a obtenu un raffinement remarquable du th\'eor\`eme de Siegel-Shidlovskii ; 
il s'agit  d'un \'enonc\'e qualitatif  
qui tient compte de chaque relation alg\'ebrique 
homog\`ene sur $\Q$ entre les  nombres $f_1(\alpha),\ldots,f_m(\alpha)$ et qui peut s'énoncer  comme suit. 

\begin{thmE2}\label{thm: liftingE} 
Sous les hypothèses du théorème E1, 
	toute relation alg\'ebrique homog\`ene sur $\Q$ entre les nombres $f_1(\alpha),\ldots,f_m(\alpha)$ est banale. 	
\end{thmE2}

Dans \cite{An1}, Andr\'e introduit la notion de $E$-op\'erateur diff\'erentiel et montre que toute $E$-fonction est annul\'ee par un tel op\'erateur. 
La d\'emonstration du th\'eor\`eme E2 par Beukers repose sur ce r\'esultat d'Andr\'e et plus pr\'ecis\'ement sur 
le fait que les seules singularit\'es d'un $E$-op\'erateur sont $0$ et $\infty$.  Andr\'e \cite{An3} a \'egalement montr\'e comment d\'eduire le th\'eor\`eme E2 du th\'eor\`eme 
de Siegel-Shidlovskii en développant une nouvelle correspondance de Galois pour les syst\`emes diff\'erentiels lin\'eaires. 

\subsection{Rappels concernant la th\'eorie des $M_q$-fonctions}

Tout comme dans le cas diff\'erentiel, il est parfois plus commode de travailler avec des 
syst\`emes $q$-mahl\'eriens lin\'eaire d'ordre $1$, plut\^ot qu'avec des \'equations  $q$-mahl\'eriennes. 
Un tel syst\`eme est de la forme :
\begin{equation}\label{eq: mahlersystem}
{\bf Y}(z^q)
=
A(z){\bf Y}(z)\,
,\quad\quad A(z) \in {\rm GL}_m(\Q(z))\, . 
\end{equation} 
Le point $\alpha\in\mathbb C$ est dit r\'egulier si la matrice $A(z)$ est bien d\'efinie et inversible en $\alpha^{q^\ell}$ pour tout $\ell\geq 0$.

L'analogue du th\'eor\`eme de Siegel-Shidlovskii dans ce cadre est un th\'eor\`eme de Ku.\ Nishioka \cite{Ni90}.

\begin{thmM1}\label{thm: trdegM}
	Soient $f_1(z),\ldots,f_m(z)$ des $M_q$-fonctions formant un vecteur solution d'un syst\`eme mahl\'erien lin\'eaire de la forme \eqref{eq: mahlersystem} 
	et $\alpha \in \Q$, $0<\vert \alpha\vert <1$, un point r\'egulier relativement \`a ce système. Alors, on a 
	$$
	{\rm deg.tr}_{\Q}(f_1(\alpha),\ldots,f_m(\alpha))={\rm deg.tr}_{\Q(z)}(f_1(z),\ldots,f_m(z))\,.
	$$
\end{thmM1}

L'analogue du th\'eor\`eme E2 a \'et\'e obtenu dans \cite{AF17} \`a partir d'une version l\'eg\`erement plus faible\footnote{La version obtenue dans \cite{PPH} ne tient pas compte du 
caract\`ere homog\`ene des relations.} d\'emontr\'ee par Philippon \cite{PPH} .

\begin{thmM2}\label{thm: liftingM} Sous les hypothèses du théorème M1, 
	 toute relation alg\'ebrique homog\`ene sur $\Q$ entre les nombres $f_1(\alpha),\ldots,f_m(\alpha)$ est banale. 	
\end{thmM2}

Dans \cite{AF17,PPH}, le théorème M2 est déduit du théorème M1 (voir également \cite{NS20} pour une démonstration analogue à celle donnée par André \cite{An3} du théorème E2). 
Une nouvelle preuve des th\'eor\`emes M1  et M2, ainsi que leurs g\'en\'eralisations \`a des fonctions mahl\'eriennes de plusieurs variables, ont \'et\'e r\'ecemment obtenues 
par les auteurs (voir \cite{AF20,AF23}). Dans cette approche, on démontre directement le théorème~M2 et on en déduit le théorème M1 par un argument déjà utilisé par Shidlovskii dans sa preuve du théorème~E1. 

\subsection{$M_q$-op\'erateurs}

Un $\sigma_q$-op\'erateur est un op\'erateur de la forme 
$$
L = a_0(z) + a_1(z)\sigma_q+\cdots + a_m(z)\sigma_q^m \in \Q[z,\sigma_q]\,.
$$
Un point $\alpha\in\mathbb C$ est dit régulier pour $L$ s'il l'est pour le syst\`eme $q$-mahl\'erien associ\'e \`a la matrice compagnon $A_L$,  
ou de fa\c con \'equivalente si 
$$
a_0(\alpha^{q^\ell})a_m(\alpha^{q^\ell})\not= 0 \quad\quad \forall \ell\geq 0\, .
$$
Dans le cas contraire, on dit que $\alpha$ est une singularit\'e de $L$.  

Dans toute la suite, $D(0,R)\subset \mathbb C$ désignera le disque ouvert de rayon $R$ et on notera $D^*(0,R):=D(0,R)\setminus\{0\}$ le disque épointé associé. 

\begin{defi}
Soit $R$, $0<R\leq 1$ un nombre r\'eel. Un $M_q$-op\'erateur de niveau $R$ est un $\sigma_q$-op\'erateur de la forme 
$$
a_0(z) + a_1(z)\sigma_q+\cdots + a_m(z)\sigma_q^m  \in \Q[z,\sigma_q]\,,
$$
tel que le polyn\^ome $a_0(z)$ n'a aucune racine dans $D^{*}(0,R)$. 
\end{defi}

\begin{lem}\label{lem: mqop}
Soit $R$, $0<R\leq 1$, un nombre r\'eel. Toute solution dans $\Q[[z]]$ d'un $M_q$-op\'erateur de niveau $R$ est analytique sur $D(0,R)$.
\end{lem}

\begin{proof}
Soit $L:=  a_0(z) + a_1(z)\sigma_q+\cdots + a_m(z)\sigma_q^m$ un $M_q$-op\'erateur de niveau $R$ et $f(z)\in \Q[[z]]$ une $M_q$-fonction annul\'ee par $L$. Raisonnons par l'absurde en supposant  que le rayon de convergence de $f$ soit \'egal \`a $\rho<R$. Notons que $\rho>0$ puisque $f(z)\in \Q[[z]]$ et que toute série formelle solution d'une équation mahlérienne linéaire est convergente. En \'ecrivant
$$
f(z)=\frac{-1}{a_0(z)} \left(a_1(z)f(z^q)+\cdots +a_m(z)f(z^{q^m})\right)\,,
$$
on constate  que le membre de droite \`a un rayon de convergence au moins \'egal \`a $\max(\rho^{1/q},R)>\rho$, puisque par d\'efinition  $a_0(z)$ ne s'annule pas sur 
$D^{*}(0,R)$. 
Ceci contredit notre hyptoh\`ese. 
\end{proof}

Le lemme \ref{lem: mqop} admet la r\'eciproque suivante. 

\begin{prop}\label{prop: mqop}
Soit $R$, $0<R\leq 1$, un nombre r\'eel. 
Toute $M_q$-fonction qui est analytique sur  $D(0,R)$ est annul\'ee par un $M_q$-op\'erateur de niveau $R$.
\end{prop}

\begin{proof}
Le \emph{d\'enominateur $q$-mahl\'erien} d'une $M_q$-fonction $f$ est d\'efini
 comme le g\'en\'erateur (choisi unitaire) de l'id\'eal principal 
$$
\left\{p(z) \in \Q[z] \ : \ p(z)f(z) \in \sum_{k\geq 1} \Q[z]f(z^{q^k})\right\}\,.
$$
Cette notion est introduite dans \cite{ABS}. Notons $\mathfrak d(z)$ le d\'enominateur $q$-mahl\'erien de $f$. 
D'apr\`es \cite[Proposition 6.4]{ABS}, si $\mathfrak d$ a une racine $\lambda$ telle que $0<\vert \lambda\vert <1$, 
alors $f$ a un rayon de convergence strictement inf\'erieur \`a un. En r\'ealit\'e, la preuve de cette proposition \'etablit explicitement que le rayon de convergence de $f$ est alors 
au plus $\vert\lambda\vert$. 
Ainsi si $f$ est analytique sur le disque $D(0,R)$, son d\'enominateur $q$-mahl\'erien ne peut s'annuler sur ce disque. On en d\'eduit l'existence d'un $M_q$-op\'erateur de niveau $R$, 
de la forme 
$$
L= \mathfrak d(z)+ a_1(z)\sigma_q+\cdots a_m(z)\sigma_q^m\,,
$$
qui annule $f$. 
\end{proof}

\begin{rem}
La proposition 6.4 de \cite{ABS} sur laquelle repose la démonstration précédente est loin d'être triviale : elle nécessite l'utilisation du théorème~M2. 
\end{rem}

\begin{prop}\label{prop: mreg}
        Soit $f$ une $M_q$-fonction analytique sur $D(0,R)$, $0<R < 1$.   Alors $f$ est annul\'ee par un $\sigma_{q}$-op\'erateur $L$ sans singularités sur $D^*(0,R)$.   
\end{prop}

\begin{proof}
D'apr\`es la proposition \ref{prop: mqop}, il existe un $M_{q}$-op\'erateur de niveau $R$ qui annule $f$. Notons 
$$
L_1:= a_0(z)+a_1(z)\sigma_{q}+\cdots +a_m(z)\sigma_{q}^m
$$
un tel op\'erateur. 
Consid\'erons un entier $s\geq 1$ tel que $a_m(z^{q^s})$ n'a aucune racine dans $D^*(0,R)$ et posons
$$
L_2 := \sigma_q^s(L_1)=a_0(z^{q^s})\sigma_{q}^s+\cdots +a_m(z^{q^s})\sigma_{q}^{m+s}\,.
$$
Puisque $L_1(f)=0$, on a \'egalement $L_2(f)=0$ et l'op\'erateur $L:=L_1+L_2$ annule donc $f$. De plus, on a 
$$
L= b_0(z)+b_1(z)\sigma_{q}+\cdots+b_{m+s}\sigma_{q}^{m+s}\,,
$$
o\`u  $b_0(z)=a_0(z)$ et $b_{m+s}(z)=a_m(z^{q^s})$.  

Soit $\alpha\in D^*(0,R)$. Comme les polyn\^omes $a_0(z)$ et $a_m(z^{q^s})$ ne s'annulent pas sur $D^*(0,R)$, on obtient que 
$$
b_0(\alpha^{q^\ell})b_{m+s}(\alpha^{q^\ell})\not=0 \,\quad\quad\forall \ell\geq 0\,.
$$
Ainsi $\alpha$ est un point r\'egulier pour $L$. 
\end{proof}

\subsection{Preuve du th\'eor\`eme \ref{thm: orbital}}

Nous pouvons \`a pr\'esent d\'emontrer notre r\'esultat principal.

\begin{proof}[D\'emonstration du th\'eor\`eme \ref{thm: orbital}] 
Commen\c cons par  le cas (E). 
Con\-sidérons des $E$-fonctions $f_1,\ldots,f_r$,  un nombre $\alpha\in \Q^*$  et supposons qu'il existe un polynôme homogène $P\in\Q[X_1,\ldots,X_r]$  tel que 
\begin{equation}\label{eq: rel}
P(f_1(\alpha),\ldots,f_r(\alpha))=0\,.
\end{equation}
D'apr\`es \cite[Th\'eor\`eme 4.2]{An1}, toute $E$-fonction $f$ est solution d'un $E$-op\'erateur diff\'erentiel. De plus, le th\'eor\`eme 4.3 de \cite{An1} stipule 
qu'un tel op\'erateur  n'a que deux singularit\'es : 
$0$ et $\infty$. Cela signifie que  $f$ est annul\'ee par un op\'erateur diff\'erentiel  de la forme 
$$
L:= a_0(z) +a_1(z)\delta+\cdots + a_{m-1}(z)\delta^{m-1}+z^\nu \delta^{(m)}\,,
$$	
où les $a_i(z)$ sont dans $\Q[z]$ et $\nu\geq 0$ est un entier.  
Ainsi, le syst\`eme diff\'erentiel associ\'e \`a la matrice compagnon $A_L$ d'une telle \'equation est \`a coefficients dans $\Q[z,1/z]$.

Pour tout $i$, $1\leq i\leq r$, il existe donc une matrice $A_i(z) \in {\mathcal M}_{m_i}(\Q[z,1/z])$ telle que 
$$
\left(\begin{array}{c}\delta(f_i) \\ \vdots \\ \delta^{m_i}(f_i)\end{array}\right)
=
A_i(z)\left(\begin{array}{c}f_i \\ \vdots \\ \delta^{m_i-1}(f_i)\end{array}\right) \, ,
$$
o\`u $m_i$ peut \^etre choisi comme le minimum des ordres des $E$-op\'erateurs annulant $f_i$. 
Notons 
$$	A(z):=A_1(z)\oplus A_2(z)\oplus\cdots \oplus A_r(z)
$$
la somme directe des matrices $A_i(z)$.  
Les fonctions $f_{i,j}(z)=\delta^j(f_i)$, $1\leq i \leq r$, $0 \leq j \leq m_i-1$, sont les coordonn\'ees d'un vecteur solution du syst\`eme diff\'erentiel associ\'e \`a la matrice $A(z)$. 
Comme $A(z)$ est \`a coefficients dans $\Q[z,1/z]$, tout point $\alpha\in\Q^*$ est r\'egulier pour ce syst\`eme.  
D'après \eqref{eq: rel} et le th\'eor\`eme E2, il existe un polyn\^ome $Q \in \Q[z,(X_{i,j})_{1\leq i \leq r,0\leq j <m_i}]$ homog\`ene en les variables $X_{i,j}$, $1\leq i \leq r$, 
$0\leq j <m_i$, tel que 
$$
Q(z,f_{i,j}(z))=0, \quad \text{ et } \quad Q(\alpha,(X_{i,j})_{1\leq i \leq r,0\leq j <m_i})=P(X_{1,0},\ldots,X_{r,0})\,.
$$
Ainsi, il suit que soit $Q\in\mathfrak I(f_1,\ldots,f_r)$ et la relation \eqref{eq: rel} est banale, soit $Q\not\in\mathfrak I(f_1,\ldots,f_r)$ et  \eqref{eq: rel} s'obtient alors par 
dégénérescence au point $\alpha$ d'une relation $\delta$-alg\'ebrique homog\`ene entre les fonctions $f_1,\ldots,f_r$.


\medskip

Passons au cas {\rm(${\rm M}_q$)}. Soient $f_1,\ldots,f_r$ des $M_q$-fonctions et $\alpha\in \Q$, $0<\vert \alpha\vert<1$, qui n'est p\^ole d'aucune de ces fonctions.  
On suppose qu'il existe $P\in\Q[X_1,\ldots,X_r]$ homog\`ene tel que 
\begin{equation}\label{eq: relm}
P(f_1(\alpha),\ldots,f_r(\alpha))=0\,.
\end{equation}

Choisissons un nombre r\'eel $R$ tel que $0<\vert\alpha\vert<R<1$. Comme les fonctions $f_i$ n'ont qu'un nombre fini de pôles sur $D(0,R)$,  il existe un polynôme $D(z)\in\Q[z]$ tel que $g_i(z):=D(z)f_i(z)$ est analytique sur $D(0,R)$ pour tout $i$, $1\leq i\leq r$, et $D(\alpha)=1$. 
D'après \eqref{eq: relm}, on a 
\begin{equation}\label{eq: relm2}
P(g_1(\alpha),\ldots,g_r(\alpha))=0\,.
\end{equation}

D'autre part, la proposition \ref{prop: mreg} assure que chaque fonction $g_i$ 
est annul\'ee par un $\sigma_{q}$-op\'erateur $L_i$, d'ordre not\'e $m_i$, qui est r\'egulier au point $\alpha$. En d\'esignant par $A_i(z)$ la matrice compagnon de l'op\'erateur $L_i$, on obtient  
que les fonctions $g_i,\sigma_{q}(g_i)\ldots,\sigma_{q}^{m_i-1}(g_i)$ forment un vecteur solution  d'un syst\`eme $q$-mahl\'erien 
\begin{equation*}
{\bf Y}(z^{q})
=
A_i(z){\bf Y}(z)
,\quad A_i(z) \in {\rm GL}_{m_i}(\Q(z))\, . 
\end{equation*} 
pour lequel  $\alpha$ est r\'egulier. 
Notons 
$$	A(z):=A_1(z)\oplus A_2(z)\oplus\cdots \oplus A_r(z)
$$
la somme directe des matrices $A_i(z)$.  
Les fonctions $f_{i,j}(z) :=\sigma_{q}^j(g_i)$, $1\leq i \leq r$, $0 \leq j \leq m_i-1$, sont les coordonn\'ees  d'un vecteur solution du syst\`eme $q$-mahl\'erien 
associ\'e \`a la matrice $A(z)$ et le point $\alpha$ est encore r\'egulier pour ce syst\`eme.   
D'après \eqref{eq: relm2} et le th\'eor\`eme M2,  il existe un polyn\^ome $Q_0 \in \Q[z,(X_{i,j})_{1\leq i \leq r,0\leq j <m_i}]$ homog\`ene en les variables 
$X_{i,j}$, $1\leq i \leq r$, 
$0\leq j <m_i$, tel que 
$$
Q_0(z,f_{i,j}(z))=0 \quad \text{ et } \quad Q_0(\alpha,(X_{i,j})_{1\leq i \leq r,0\leq j <m_i})=P(X_{1,0},\ldots,X_{r,0})\,.
$$
Définissons $Q_1\in \Q[z,(X_{i,j})_{1\leq i \leq r,0\leq j <m_i}]$ par  
$$Q_1(z,(X_{i,j})_{1\leq i \leq r,0\leq j <m_i}):=Q_0(z,(D(z^{q^j})X_{i,j})_{1\leq i \leq r,0\leq j <m_i})\,.$$  
On a alors 
$$Q_1(z,\sigma_q^j(f_i)(z))=0 \;\;\mbox{ et }\;\; Q_1(\alpha,\sigma_q^j(f_i)(\alpha))=P(f_1(\alpha),\ldots,f_r(\alpha))\,,$$ 
puisque $D(\alpha)=1$. 
On a donc que soit $Q_1\in\mathfrak I(f_1,\ldots,f_r)$ et la relation \eqref{eq: relm} est banale, soit $Q\not\in\mathfrak I(f_1,\ldots,f_r)$ et  \eqref{eq: relm} s'obtient alors par 
dégénérescence au point $\alpha$ d'une relation $\sigma_q$-alg\'ebrique homog\`ene entre les fonctions $f_1,\ldots,f_r$.
\end{proof} 

\subsection{Remarque sur les preuves des théorèmes E2 et M2} 

Dans \cite{Be06}, Beukers prouve le théorème E2 de la façon suivante. 
D'une part,  il montre que pour toute $E$-fonction $f$ et tout $\alpha\in \Q^*$ :
\begin{eqnarray}\label{eq: implies}
f(\alpha)=0 \implies & \mbox{$\alpha$ est une singularité de }
\\
\nonumber & \mbox{l'opérateur différentiel minimal annulant $f$}.
\end{eqnarray}
 D'autre part, en supposant par l'absurde que la conclusion du théorème E2 est fausse, il donne une construction élémentaire et générale d'une 
 $E$-fonction qui contredit \eqref{eq: implies}. 
Il ressort ainsi que le point clé pour démontrer le théorème E2 est d'obtenir l'implication \eqref{eq: implies}. 
Celle-ci découle du fait que pour toute $E$-fonction $f$ et tout $\alpha\in \Q^*$, $f$ est annulée par un opérateur différentiel régulier en $\alpha$ ; propriété 
elle-même garantie par l'existence  d'un $E$-opérateur annulant $f$ (cf. \cite[Th\'eor\`eme 4.2]{An1}). 

En adaptant l'argument de \cite{Be06}, on obtient que le théorème M2 découle du fait que 
pour toute $M_q$-fonction $f$ et tout $\alpha\in\Q$, $0<\vert \alpha\vert <1$: 
\begin{eqnarray}\label{eq: implies2}
f(\alpha)=0 \implies & \mbox{$\alpha$ est une singularité du }
\\
\nonumber & \mbox{$\sigma_q$-opérateur minimal annulant $f$}.
\end{eqnarray}
Montrons que cette propriété est une conséquence de la proposition \ref{prop: mqop}. 
Soit  
\begin{equation}\label{eq: eqminf}
a_0(z)f(z)+\cdots+a_n(z)f(z^{q^n})=0 \,,
\end{equation}
l'équation minimale de $f$. 
Supposons par l'absurde que $f(\alpha)=0$ et que $\alpha$ est régulier pour cette équation, c'est-\`a-dire que 
\begin{equation}\label{eq: regalpha}
a_0(\alpha^{q^\ell})a_n(\alpha^{q^\ell})\not=0 \,, \quad\quad \forall\ell\geq 0. 
\end{equation}
Posons $g(z):=f(z)/(z-\alpha)$. D'après \eqref{eq: eqminf}, $g$ est annulée par le $\sigma_q$-opérateur 
\begin{equation}\label{eq: eqming}
L_1:=a_0(z)(z-\alpha) +\cdots +a_n(z)(z^{q^n}-\alpha)\sigma_q^n=0\,.
\end{equation}
Soit $R$ tel que $\vert \alpha\vert <R<1$. Comme $g$ est bien définie en $\alpha$, il existe $P(z)\in\Q[z]$ tel que $P(z)g(z)$ est analytique sur $D(0,R)$ et $P(\alpha)=1$. 
D'après la proposition \ref{prop: mqop}, 
$P(z)g(z)$ est annulée par un $M_q$-opérateur de niveau $R$. On en déduit  l'existence d'un $\sigma_{q}$-opérateur  
\begin{equation}\label{eq: g2}
L_2:=b_0(z)+b_1(z)\sigma_{q}+\cdots+b_m(z)\sigma_{q}^m 
\end{equation}
qui annule $g$ et tel que $b_0(\alpha)\not=0$. 
Considérons un tel opérateur avec $m$ minimal. Notons que $m\geq n$. Posons $$L_3:=b_m(z)\sigma^{m-n}L_1 - a_n(z^{q^{m-n}})(z^{q^m}-\alpha)L_2 \,.$$   
On obtient que $L_3$ annule $g$ et est d'ordre strictement inférieur à $m$. De plus, on vérifie que le terme constant de $L_3$ vaut 
$$
a_n(z^{q^{m-n}})(z^{q^m}-\alpha)b_0(z)
$$ 
si $m>n$ et 
$$
a_n(z^{q^{m-n}})(z^{q^m}-\alpha)b_0(z) - b_m(z)a_0(z)(z-\alpha)
$$
lorsque $m=n$. Dans les deux cas,  \eqref{eq: regalpha} assure que ce terme ne s'annule pas en $\alpha$.  
Cela contredit la minimalité de $m$ et prouve donc \eqref{eq: implies2}.

Il faut prendre garde au fait que nous n'obtenons pas pour autant une preuve \og sans transcendance\fg{} du théorème M2 \`a la façon de \cite{An2,Be06}. En effet, la démonstration que nous avons donnée de la proposition \ref{prop: mqop} repose sur le théorème M2 et l'argument serait  donc circulaire. Nous attirons simplement l'attention sur le fait que, \emph{in fine}, le théorème M2 et la proposition \ref{prop: mqop} sont essentiellement équivalents, de même que le théorème E2 et le fait que pour toute $E$-fonction $f$ et tout $\alpha\in\Q^*$, $f$ est annulée par un opérateur différentiel régulier en $\alpha$. 
Dans le cas des $E$-fonctions, ce sont les résultats profonds obtenus sur les $G$-opérateurs qui 
permettent d'obtenir  \og miraculeusement\fg{} l'existence de tels opérateurs, via la notion de $E$-opérateur (cf. \cite{An1}).  
Dans le cas des $M_q$-fonctions, nous ne savons pour le moment pas comment obtenir l'existence de ces opérateurs indépendamment du théorème M2.

\section{Détermination effective des relations algébriques}\label{sec: eff}

Dans cette section, nous discutons brièvement la question suivante, ainsi que son analogue pour les $M_q$-fonctions : étant donnés  $\alpha\in\Q^*$ et des $E$-fonctions $f_1,\ldots,f_r$, peut-on déterminer les relations algébriques sur $\Q$ entre les nombres $f_1(\alpha),\ldots,f_r(\alpha)$ ? Il s'agit donc de déterminer explicitement une base de l'idéal 
$$
\mathfrak I(f_1(\alpha),\ldots,f_r(\alpha) ):= \{P\in \Q[X_1,\ldots,X_r] : P(f_1(\alpha),\ldots,f_r(\alpha))=0\}\,.
$$
L'article \cite{FR19} décrit un algorithme pour répondre à cette question dans le cas des $E$-fonctions. 
Nous décrivons une approche alternative fondée sur le théorème \ref{thm: orbital} qui, nous semble-t-il, est particulièrement limpide. L'utilisation du théorème \ref{thm: orbital} permet à nouveau de traiter les $E$- et les $M_q$-fonctions de façon unifiée. Notons que ces deux approches souffrent d'une même limite liée à l'utilisation 
de l'algorithme de Hrushovski-Feng \cite{Fe15} (ou de Feng \cite{Fe18} pour les $M_q$-fonctions).

\begin{rem}
Comme dans \cite{AR18}, on supposera que chaque fonction $f_i$ est  donnée par un opérateur différentiel linéaire $L_i$ qui l'annule et 
les premiers coefficients de son développement de Taylor à l'origine. On supposera que le nombre de coefficients connus est suffisamment grand pour que $f_i$ soit uniquement déterminée 
par ces données. 
\end{rem}

 \'Etant donn\'e un id\'eal $\mathfrak I$ de $\Q(z)[X_1,\ldots,X_s]$, $s\geq 1$, on note ${\rm ev}_\alpha(\mathfrak I)$ l'id\'eal de $\Q[X_1,\ldots,X_s]$ obtenu 
en \'evaluant les \'el\'ement de $\mathfrak I\cap \Q[z,X_1,\ldots,X_s]$ en $z=\alpha$. 

Nous traitons d'abord le cas des $E$-fonctions. 

\begin{itemize}

\item[{\rm 1.}] La première étape consiste à déterminer pour chaque $f_i$ un entier $m_i$ tel que $f_i$ est annulée par un $E$-opérateur d'ordre au plus $m_i$. 
On suit l'approche d'André \cite{An1}. Si 
$f(z)=\sum_{n=0}^\infty \frac{a_n}{n!}z^n$, on dit que $g(z):=\sum_{n=0}^\infty a_nz^n$ est la $G$-fonction associée à $f$. \`A partir de l'opérateur $L_i$, on peut déterminer un opérateur différentiel 
$L'_i$ annulant la $G$-fonction associée à $f_i$ en utilisant la transformée de Fourier-Laplace. \`A partir de ce premier opérateur on peut obtenir un opérateur différentiel d'ordre minimal annulant cette $G$-fonction (voir \cite{BRS}). C'est un $G$-opérateur.  \`A partir de ce $G$-opérateur, on détermine facilement un $E$-opérateur explicite annulant $f_i$ (cf. \cite[Sections 2 et 5]{An1}). On note $m_i$ l'ordre de cet opérateur et on pose  ${\bf m}:=(m_1-1,\ldots,m_r-1)$. 

\item[{\rm 2.}]  La démonstration du théorème \ref{thm: orbital} montre alors que 
$$
\mathfrak I(f_1(\alpha),\ldots,f_r(\alpha) ) ={\rm ev}_\alpha(\mathfrak I^{\delta}_{\bf m}(f_1,\ldots,f_r))\cap \Q[X_1,\ldots, X_r]\,, 
$$
où $\mathfrak I^{\delta}_{\bf m}(f_1,\ldots,f_r):=\mathfrak I^{\delta}(f_1,\ldots,f_r)\cap \Q(z)[(X_{i,j})_{1\leq i \leq r,0\leq j <m_i}]$ et où l'on a identifié 
les variables $X_{i,0}$ et $X_i$, $1\leq i\leq r$.

\item[{\rm 3.}] Supposons que l'on dispose d'un ensemble explicite de générateurs  de $\mathfrak I^{\delta}_{\bf m}(f_1,\ldots,f_r)$. En les évaluant  en $z=\alpha$, on obtient 
un ensemble explicite de générateurs de l'idéal 
${\rm ev}_\alpha(\mathfrak I^{\delta}_{\bf m}(f_1,\ldots,f_r))$. La détermination explicite d'une base de l'idéal 
$${\rm ev}_\alpha(\mathfrak I^{\delta}_{\bf m}(f_1,\ldots,f_r))\cap \Q[X_1,\ldots, X_r]$$ 
s'obtient alors classiquement en utilisant la théorie des bases de Gröbner (voir \cite[Section 6.2]{BW} pour l'algorithme correspondant). 

\end{itemize} 

Il ressort de cette analyse que la connaissance d'une base de $\mathfrak I^{\delta}_{\bf m}(f_1,\ldots,f_r)$ suffit à déterminer efficacement 
une base de  $\mathfrak I(f_1(\alpha),\ldots,f_r(\alpha) )$. Toute la difficulté réside donc dans l'obtention de cette première. 
Les fonctions $\delta^j(f_i)$, $1\leq i\leq r$, $1\leq j< m_i$, étant liées par un système différentiel linéaire, l'algorithme de 
Hrushovski-Feng \cite{Fe15} permet en principe d'obtenir une base de $\mathfrak I^{\delta}_{\bf m}(f_1,\ldots,f_r)$. 
Cependant, cet algorithme semble difficile à mettre en \oe uvre  du point de vue pratique ; c'est là le talon d'Achille de cette méthode.

Dans le cas des $M_q$-fonctions, on peut suivre rigoureusement la même approche, en substituant l'algorithme de Feng \cite{Fe18} à celui de Hrushovski-Feng.  
Notons que cet algorithme est tout aussi peu commode à 
utiliser dans la pratique.  La seule différence réside dans la première étape, à savoir la détermination explicite d'un vecteur ${\bf m}:=(m_1-1,\ldots,m_r-1)$ tel que 
$$
\mathfrak I(f_1(\alpha),\ldots,f_r(\alpha) ) ={\rm ev}_\alpha(\mathfrak I^{\sigma_q}_{\bf m}(f_1,\ldots,f_r))\cap \Q[X_1,\ldots, X_r]\,,
$$
où $\mathfrak I^{\sigma_q}_{\bf m}(f_1,\ldots,f_r):=\mathfrak I^{\sigma_q}(f_1,\ldots,f_r)\cap \Q(z)[X_{i,j})_{1\leq i \leq r,0\leq j <m_i}]$. 
Une analyse de la preuve du théorème \ref{thm: orbital} montre que s'il existe un entier $\ell\geq 1$ tel que pour chaque $i$, $1\leq i\leq r$, 
on dispose d'un système $q^\ell$-malhérien linéaire régulier en  
$\alpha$ et possédant un vecteur solution formé de la fonction $f_i$ et de fonctions de la forme $\sigma_{q^\ell}^j(f_i)$ pour $j\leq c_i$, alors on peut choisir $m_i:=\ell c_i+1$. 

Voici comment construire explicitement de tels systèmes. Partant d'un $\sigma_q$-opérateur annulant $f_i$, on détermine un $\sigma_q$-opérateur d'ordre minimal annulant $f_i$ (cf. \cite{AF18}). Notons $n_i$ son ordre. D'après  le théorème 1.10 de \cite{AF17}, on peut  itérer le système compagnon associé, puis  utiliser  
la technique de dédoublement  décrite dans la preuve du lemme 5.2 de \cite{AF17}, pour obtenir un système $q^\ell$-mahlérien pour un certain $\ell$ explicite tel que $\alpha$ est régulier pour ce système. En outre, il suffit de choisir  $\ell$ suffisamment grand par rapport à $q$, au module de $\alpha$, et au maximum des degrés des polynômes de l'équation minimale de $f_i$\footnote{Plus précisément, il suffit de choisir $\ell$ de sorte que $\alpha^{q^\ell}$ soit régulier pour l'équation minimale de $f_i$. Si $R<1$ est fixé, on choisit $\ell$ tel que $a(z^{q^\ell})$ n'a pas de racine dans $D^*(0,R)$, pour tout polynôme $a(z)$ qui est coefficient de l'équation minimale de $f_i$. Ce choix de $\ell$ conviendra alors pour tout $\alpha\in D^*(0,R)$.}. 
On peut ainsi trouver explicitement  un entier 
$\ell$ qui convient pour toutes les fonctions $f_i$. Il suffit alors de prendre $m_i:= 2\ell+n_i$. 

\begin{rem}\label{rem: comp}
Dans le cas (E), nous avons vu que le vecteur ${\bf m}$ peut être déterminé indépendamment de $\alpha$. Dans le cas (${\rm M}_q$), on peut déterminer ${\bf m}$ qui convient pour tout $\alpha\in D^*(0,R)$ lorsque $R<1$ est préalablement fixé, mais il n'existe pas nécessairement un vecteur ${\bf m}$ qui convient pour tous les points du disque unité (cf. exemple \ref{ex: tm3}). 
\end{rem}

\section{Sporadicit\'e des relations non banales}\label{sec: spor}

La proposition suivante \'etablit que dans le cadre des $E$- et des $M_q$-fonctions les relations non banales sont nécessairement sporadiques. 

\begin{prop}\label{prop: sporadique} On a les deux r\'esultats suivants. 
	
	\begin{itemize}
\item[{\rm (E)}] Soient $f_1(z),\ldots,f_r(z)$ des $E$-fonctions. Il existe un ensemble fini $\mathcal S\subset \Q$ tel que l'existence d'une relation non banale entre 
$f_1(\alpha),\ldots,f_r(\alpha)$ implique que  $\alpha\in\mathcal S$. 

\item[{\rm(${\rm M}_q$)}] Soient $f_1(z),\ldots,f_r(z)$ des $M_q$-fonctions et $R$ un nombre réel, $0<R<1$.  
Il existe un ensemble fini $\mathcal S_R\subset \Q$ tel que l'existence d'une relation non banale entre 
$f_1(\alpha),\ldots,f_r(\alpha)$ avec $\alpha\in D^*(0,R)$ implique que $\alpha\in\mathcal S_R$. 
 \end{itemize}
 \end{prop}

\begin{rem}
En général, les relations non banales entres des fonctions analytiques n'ont aucune raison d'être sporadiques, comme le montre l'exemple suivant. 
Consid\'erons la fonction d\'efinie par 
$$
f(z):=\sum_{n=1}^{\infty} c_k\left(\prod_{i=1}^kP_i(z)\right) \,,
$$
o\`u $P_1,P_2,\ldots$ est une \'enum\'eration des polyn\^omes \`a coefficients entiers et $(c_k)_{k\geq 1}$ une suite d\'ecroissante de nombres rationnels non nuls tendant vers $0$. 
Si la d\'ecroissance de $c_k$ est suffisamment rapide, alors $f$ est une fonction enti\`ere, ce que nous supposerons désormais. Ainsi, comme $f$ n'est pas un polyn\^ome, elle est donc transcendante et les   fonctions $1$ et $f$ sont lin\'eairement ind\'ependantes sur $\Q(z)$.  
Par construction, pour tout $\alpha\in\Q^*$, on a $f(\alpha)\in\Q$, de sorte que la relation lin\'eaire sur $\Q$ donn\'ee par $f(\alpha)\times 1 - 1\times f(\alpha)=0$ est non banale relativement \`a $\{1,f(z)\}$. 
\end{rem}

D\'emontrons d'abord le lemme suivant. Nous conservons les notations introduites dans la section \ref{sec: eff}.

\begin{lem}\label{lem:elimination_ideal}
	Soit $\mathfrak I$ un id\'eal de $ \Q(z)[X_1,\ldots,X_m]$ et posons $\mathfrak I_r :=\mathfrak I\cap \Q(z)[X_1,\ldots,X_r]$. Alors, il existe un ensemble fini $\mathcal S\subset \Q$ tel que, pour tout $\alpha$ dans 
	$\Q\setminus \mathcal S$, on a 
	$$
	{\rm ev}_\alpha(\mathfrak I) \cap \Q[X_1,\ldots,X_r] = {\rm ev}_\alpha(\mathfrak I_r)\,.
	$$
\end{lem}

\begin{proof}
  L'argument repose sur l'utilisation de bases de Gröbner. Nous renvoyons  à \cite{BW} pour les résultats classiques concernant cette notion. 
	
	On considère une base de Gröbner $\mathcal G=\{P_1,\ldots,P_s\}$ relativement \`a un ordre monomial $\succ$ qui élimine $X_{r+1},\ldots,X_m$. Pour un tel ordre, tout monôme contenant une des variables $X_{r+1},\ldots,X_m$ est plus grand qu'un monôme composé uniquement des variables  $X_{1},\ldots,X_r$. De plus, l'ordre est compatible avec la multiplication. Pour un polynôme $P$, on note ${\rm lm}(P)$ le monôme le plus grand pour la relation d'ordre totale $\succ$ et ${\rm lc}(P)\in \Q(z)$ le coefficient correspondant. On a 
	$$
	P={\rm lc}(P){\rm lm}(P) +  P'\,
	$$
	où $P'$ est un polynôme dont les monômes sont tous plus petits que ${\rm lm}(P)$. 
	Soit $\mathcal S\subset \Q$ l'ensemble des  racines des polynômes ${\rm lc}(P_i)$, $1\leq i\leq s$. On va montrer que si $\alpha\in\Q\setminus \mathcal S$, alors 
	$$
	{\rm ev}_\alpha(\mathfrak I) \cap \Q[X_1,\ldots,X_r] = {\rm ev}_\alpha(\mathfrak I_r)\,.
	$$ 
	Supposons par l'absurde que ce ne soit pas le cas. Considérons $\alpha\in\Q\setminus \mathcal S$ tel qu'il 
existe  $P \in \mathfrak I\cap \Q[z,X_1,\ldots,X_m]$ pour lequel  
	$$
	{\rm ev}_{\alpha}(P) \in \Q[X_1,\ldots,X_r] \setminus {\rm ev}_\alpha(\mathfrak I_r) \,.
	$$
	On choisit $P$ de sorte que ${\rm lm}(P)$ soit minimal parmi les polynômes ayant cette propriété.  
	Comme $\mathcal G$ est une base de Gröbner, il existe un $P_i$ dont le monôme dominant divise celui de $P$. De plus, comme $\alpha\not\in\mathcal S$, on a 
	$ {\rm ev}_\alpha({\rm lc}(P_i))\not=0$. On pose
	$$
Q := \left({\rm lc}(P_i)P - \frac{{\rm lc}(P){\rm lm}(P)}{ {\rm lm}(P_i)}P_i\right)/{\rm ev}_\alpha({\rm lc}(P_i)) \in \Q[z,X_1,\ldots,X_m]\,.
	$$
	Comme $P \notin \mathfrak I_r$,  ${\rm lm}(P)$ contient au moins une des variables $X_{r+1},\ldots,X_m$. Comme ${\rm ev}_{\alpha}(P) \in \Q[X_1,\ldots,X_r]$, le coefficient dominant de $P$ doit s'annuler en $\alpha$. On obtient donc que ${\rm ev}_{\alpha}(Q)={\rm ev}_{\alpha}(P)$. Cependant, le monôme dominant de $Q$ est strictement plus petit que celui de $P$, ce qui contredit la définition de $P$. 
\end{proof}

\begin{proof}[Démonstration de la Proposition \ref{prop: sporadique}]
	Commençons par le cas (E). Soit ${\bf m}\in\mathbb N^r$ tel que pour tout $\alpha\in\Q^*$, on a  
	$$\mathfrak I(f_1(\alpha),\ldots,f_r(\alpha))= {\rm ev}_\alpha(\mathfrak I_{\bf m}^\delta(f_1,\ldots,f_r))\cap \Q[X_1,\ldots,X_r]\,.$$ 
	Nous avons vu qu'un tel ${\bf m}$ existe (cf. Section \ref{sec: eff}).  
	Posons $\mathfrak I:= \mathfrak I_{\bf m}^\delta(f_1,\ldots,f_r)$ et $\mathfrak I_r:=\mathfrak I\cap \Q[X_1,\ldots,X_r]$. 
	S'il existe une relation  non banale entre  les nombres $f_1(\alpha),\ldots,f_r(\alpha)$, alors il existe un polynôme $Q$ tel que
	$${\rm ev}_\alpha(Q)  \in {\rm ev}_\alpha(\mathfrak I) \cap \Q[X_1,\ldots,X_r]\setminus {\rm ev}_\alpha(\mathfrak I_r)$$ et donc 
	 $${\rm ev}_\alpha(\mathfrak I) \cap \Q[X_1,\ldots,X_r] \not= {\rm ev}_\alpha(\mathfrak I_r)\,.$$ D'après le lemme \ref{lem:elimination_ideal}, il n'existe qu'un nombre fini de tels $\alpha$.

	Le cas (${\rm M}_q$) est similaire. On choisit ${\bf m}\in\mathbb N^r$ tel que 
	$$\mathfrak I(f_1(\alpha),\ldots,f_r(\alpha))= {\rm ev}_\alpha(\mathfrak I_{\bf m}^{\sigma_q}(f_1,\ldots,f_r))\cap \Q[X_1,\ldots,X_r]\,.$$ 
	La différence est que l'existence d'un tel ${\bf m}$ est seulement assurée pour tout $\alpha\in \Q^*\cap D(0,R)$ lorsque $R<1$ est fixé, mais qu'il n'est pas toujours possible de choisir ${\bf m}$ indépendamment de $R$ (cf. remarque \ref{rem: comp}).  
	\end{proof}

\section{Structure des idéaux des relations $\delta$- et $\sigma_q$-alg\'ebriques}\label{sec: ideal}

Dans cette section, nous décrivons la structure des idéaux $\mathfrak I^\delta(f_1,\ldots,f_r)$ et $\mathfrak I^{\sigma_q}(f_1,\ldots,f_r)$ 
qui interviennent dans le théorème \ref{thm: orbital}.

\subsection{Le cas des $E$-fonctions} Soient $f_1,\ldots,f_r$ des $E$-fonctions. Pour tout $i$, $1\leq i\leq r$, notons 
$$
a_{i,0}(z)+\cdots + a_{i,m_{i-1}}(z)\delta^{m_i-1}+\delta^{m_i}
$$
l'opérateur différentiel d'ordre minimal annulant $f_i$, où $a_{i,j}(z)\in\Q(z)$ pour tout $j$, $1\leq j\leq m_i-1$. 
En posant 
$$
	L_i:=a_{i,0}(z)X_{i,0}+\cdots + a_{i,m_{i-1}}(z)X_{i,m_{i-1}}+X_{i,m_i}\,,
	$$
	on a donc $L_i\in  \mathfrak I^\delta(f_1,\ldots,f_r)$. 
	Posons ${\bf m}:=(m_1-1,\ldots,m_r-1)$ et 
	$$
	\mathfrak I_{\bf m}^{\delta}(f_1,\dots,f_r) := \mathfrak I^\delta(f_1,\ldots,f_r)\cap \Q(z)[(X_{i,j})_{1\leq i \leq r, 0\leq j< m_i }] \,. 
	$$

Rappelons que si $R$ est un anneau muni d'une dérivation $\partial$, un $\partial$-idéal $\mathcal I$ de $R$ est un idéal de $R$ tel que $\partial(\mathcal I)\subset \mathcal I$. 
Le $\partial$-idéal engendré par un sous-ensemble $\mathcal S$ de $R$ est  le plus petit $\partial$-idéal de $R$ contenant $\mathcal S$. 
L'anneau $\Q(z)[(X_{i,j})_{1\leq i \leq r,j\in \N}]$ 
peut être muni de la dérivation $\delta$ qui agit classiquement sur $\Q(z)$ et telle que $\delta(X_{i,j})=X_{i,j+1}$. Muni de cette dérivation, on vérifie que $\mathfrak I^\delta(f_1,\ldots,f_r)$ 
est un $\delta$-idéal. L'anneau de polynômes $\Q(z)[(X_{i,j})_{1\leq i \leq r,j\in \N}]$ n'est pas noethérien et l'idéal $\mathfrak I^\delta(f_1,\ldots,f_r)$ n'est pas toujours finiment engendré ; il l'est toutefois en tant que $\delta$-idéal.  
Le résultat suivant précise un ensemble de générateurs.

\begin{prop}\label{prop: deltaideal}
En tant que $\delta$-idéal, $\mathfrak I^\delta(f_1,\ldots,f_r)$ est engendré par les éléments de $\mathfrak I^{\delta}_{\bf m}(f_1,\ldots,f_r)$ et les formes linéaires $L_1,\ldots,L_r$.
\end{prop} 

\begin{proof}
 Notons $\mathfrak J$ le $\delta$-idéal engendré par les éléments de l'idéal $\mathfrak I^{\delta}_{\bf m}(f_1,\ldots,f_r)$ et les formes linéaires $L_1,\ldots,L_r$. Il est clair que 
 $$\mathfrak J \subset \mathfrak I^{\delta}(f_1,\ldots,f_r)\,.$$ 
 Montrons l'inclusion réciproque. 
 Soit $Q \in \mathfrak I^\delta(f_1,\ldots,f_r)$. Considérons l'ordre partiel  sur $\N^r$ défini par ${\bf m}\leq {\bf n}$ si $m_i\leq n_i$ pour tout $i$, $1\leq i\leq r$. 
La \emph{multi-profondeur} de $Q$ est définie comme le $r$-uplet minimal $(n_1,\ldots,n_r)$ tel que le support de $Q$ est inclus dans  $\{(i,j) : 1\leq i \leq r,0\leq j\leq n_i\}$. 

Nous allons raisonner par récurrence multiple sur  $(n_1,\ldots,n_r)$. Si $n_i < m_i$ pour tout $i$, alors $Q \in \mathfrak I^{\delta}_{\bf m}(f_1,\ldots,f_r)$ et il n'y a rien à prouver. 
Supposons à présent que tout polynôme de $\mathfrak I^\delta(f_1,\ldots,f_r)$ de multi-profondeur au plus $(n_1,n_2,\ldots,n_r)$ appartient \`a $\mathfrak J$ et montrons qu'il en est de 
même pour tout 
polynôme de $\mathfrak I^\delta(f_1,\ldots,f_r)$ de multi-profondeur au plus $(n_1+1,n_2,\ldots,n_r)$, où $n_1+1\geq m_1$. L'argument pour les polynômes de $\mathfrak I^\delta(f_1,\ldots,f_r)$ 
de multi-profondeur au plus 
$(n_1,\ldots,n_j+1,\ldots,n_r)$ est identique.

On note $\X^{\circ}$ l'ensemble des variables $X_{i,j}$, $1\leq i \leq r$, $1 \leq j \leq n_i$. 
Soit $Q$  un élément de  $\mathfrak I^\delta(f_1,\ldots,f_r)$ de multi-profondeur au plus $(n_1+1,n_2,\ldots,n_r)$. On peut écrire 
 $$
 Q =: \sum_{\nu=0}^s M_\nu(z,\X^\circ)X_{1,n_{1}+1}^\nu\,,
 $$
 où les $M_\nu(z,\X^\circ)$ appartiennent à $\Q(z)[\X^{\circ}]$ et $s\geq 0$.  
 On raisonne par récurrence sur l'entier $s$ pour montrer que $Q\in\mathfrak J$. 
 Si $s=0$, alors $Q$ est de multi-profondeur au plus $(n_1,n_2,\ldots,n_r)$ est la première hypothèse de récurrence implique donc que $Q\in \mathfrak J$. 
 Supposons à présent que $s\geq 1$ et que le résultat est connu pour tout élément de $\mathfrak I^\delta(f_1,\ldots,f_r)$ de multi-profondeur au plus $(n_1+1,n_2,\ldots,n_r)$ et de 
 degré au plus $s-1$ en  $X_{1,n_{1}+1}$. 
Posons
 $$
 P :=  M_s(z,\X^\circ)X_{1,n_{1}+1}^{s-1}\delta^{n_{1}+1-m_{1}}(L_{1}) \,.
 $$
Comme $L_1 \in \mathfrak J$ et que $\mathfrak J$ est un $\delta$-idéal, on a $P \in \mathfrak J\subset \mathfrak I^\delta(f_1,\ldots,f_r)$. 
De plus, comme le coefficient de  $X_{1,n_{1}+1}$ 
dans $\delta^{n_{1}+1-m_{1}}(L_{1})$ est égal à  $1$, le coefficient de  $X_{1,n_{1}+1}^s$ dans $P$ est égal à $M_s(z,\X^\circ)$.  
On obtient que $Q-P\in  \mathfrak I^\delta(f_1,\ldots,f_r)$ est un polynôme de multi-profondeur au plus $(n_1+1,n_2,\ldots,n_r)$ et  de degré au plus $s-1$ en  
la variable $X_{1,n_1+1}$. Par hypothèse de récurrence, on a $Q-P\in \mathfrak J$. 
Comme $P\in \mathfrak J$, 
on en déduit que $Q\in\mathfrak J$ comme souhaité. 
\end{proof}

\subsection{Le cas des $M_q$-fonctions} Soient $f_1,\ldots,f_r$ des $M_q$-fonctions. Pour tout $i$, $1\leq i\leq r$, notons 
$$
a_{i,0}(z)+\cdots + a_{i,m_{i-1}}(z)\sigma_q^{m_i-1}+\sigma_q^{m_i}
$$
le $\sigma_q$-opérateur d'ordre minimal annulant $f_i$, où $a_{i,j}(z)\in\Q(z)$ pour tout $j$, $1\leq j\leq m_i-1$. 
En posant 
$$
	L_i:=a_{i,0}(z)X_{i,0}+\cdots + a_{i,m_{i-1}}(z)X_{i,m_{i-1}}+X_{i,m_i}\,,
	$$
	on a donc $L_i\in  \mathfrak I^{\sigma_q}(f_1,\ldots,f_r)$. 
	Posons ${\bf m}:=(m_1-1,\ldots,m_r-1)$ et 
	$$
	\mathfrak I_{\bf m}^{\sigma_q}(f_1,\dots,f_r) := \mathfrak I^{\sigma_q}(f_1,\ldots,f_r)\cap \Q(z)[(X_{i,j})_{1\leq i \leq r, 0\leq j< m_i }] \,. 
	$$

Rappelons que si $R$ est un anneau muni d'un endomorphisme $\phi$, un $\phi$-idéal $\mathcal I$ de $R$ est un idéal de $R$ tel que $\phi(\mathcal I)\subset \mathcal I$. 
Le $\phi$-idéal engendré par un sous-ensemble $\mathcal S$ de $R$ est par définition le plus petit $\phi$-idéal de $R$ contenant $\mathcal S$. 
L'anneau $\Q(z)[(X_{i,j})_{1\leq i \leq r,j\in \N}]$ peut être muni de l'endomorphisme $\sigma_q$ qui agit  comme précédemment sur $\Q(z)$ et tel que $\sigma_q(X_{i,j})=X_{i,j+1}$. Cette définition implique que $\mathfrak I^{\sigma_q}(f_1,\ldots,f_r)$ 
est un $\sigma_q$-idéal. \`A nouveau, $\mathfrak I^{\sigma_q}(f_1,\ldots,f_r)$ est finiment engendré en tant que $\sigma_q$-idéal et  
on obtient  l'analogue de la proposition \ref{prop: deltaideal}, dont la preuve est par ailleurs identique. 

\begin{prop}\label{prop: sigmaqideal}
En tant que $\sigma_q$-idéal, $\mathfrak I^{\sigma_q}(f_1,\ldots,f_r)$ est engendré par les éléments de $\mathfrak I^{\sigma_q}_{\bf m}(f_1,\ldots,f_r)$ et les formes linéaires $L_1,\ldots,L_r$.
\end{prop} 

\section{Descente}\label{sec: descente}

Dans cette section, nous montrons comment le corollaire \ref{coro: descente} découle du théorème \ref{thm: orbital} puis nous donnons deux conséquences de ce résultat. 

Rappelons tout d'abord le lemme suivant, qui correspond au lemme 5.3 de \cite{AF17}. 
L'\'enonc\'e donn\'e ici est l\'eg\`erement modifi\'e, mais la preuve reste identique. 

\begin{lem}\label{lem:descent}
	Soient ${\mathbb K}\subset \Q$ un corps de nombres, $h_1(z),\ldots,h_r(z)\in {\mathbb  K}[[z]]$ et $\alpha\in \Q$.  
	Supposons qu'il existe   $w_1(z),\ldots,w_r(z)\in \Q[z]$ tels que 
	\begin{equation}\label{eq:P2-lemdesc}
	w_1(z)h_1(z) + \ldots + w_r(z)h_r(z) =0 \, ,
	\end{equation}
	avec $w_i(\alpha)=0$  pour tout indice $i$ dans un ensemble $ \mathcal I\subset \{1,\ldots,r\}$ et $w_{i_0}(\alpha)\not= 0$ pour un certain indice 
	$i_0\in\{1,\ldots,r\}\setminus \mathcal I$.   
	Alors, il existe $w'_1(z),\ldots,w'_r(z)\in {\mathbb K}[z]$, tels que 
	$$
	w'_1(z)h_1(z) + \ldots + w'_r(z)h_r(z) =0 \, ,
	$$
	avec $w'_i(\alpha)=0$ pour tout indice $i\in \mathcal I$ et $w'_{i_0}(\alpha)\not=0$. 
\end{lem}

\begin{proof}[D\'emonstration du corollaire \ref{coro: descente}]
        Nous d\'emontrons seulement le cas (E), la preuve du cas {\rm(${\rm M}_q$)} étant identique. 
        
	Soit $\mathbb K\subset \Q$ un corps. Supposons que les nombres $\xi_1,\ldots,\xi_r\in {\bf E}_\mathbb K$ sont lin\'eairement d\'ependants sur $\Q$.  
	Quitte \`a choisir un sous-corps de $\mathbb K$, on peut sans perte de g\'en\'eralit\'e supposer que $\mathbb K$ est un corps de nombres (puisque les coefficients d'une $E$-fonction engendre toujours une extension finie de $\mathbb Q$).  Il existe donc des $E$-fonctions $f_1(z),\ldots,f_r(z)\in\mathbb K[[z]]$ telles que $\xi_1=f_1(1),\ldots,\xi_r=f_r(1)$. D'apr\`es le th\'eor\`eme \ref{thm: orbital}, il existe une  relation $\delta$-lin\'eaire  sur $\Q(z)$ entre les fonctions $f_1(z),\ldots,f_r(z)$ qui se sp\'ecialise en $z=1$ en une relation non triviale entre les nombres $\xi_1,\ldots,\xi_r$. Notons que les d\'eriv\'ees successives des $f_i$ restent \`a coefficients dans $\mathbb K$. D'apr\`es le lemme \ref{lem:descent}, il existe donc une  relation $\delta$-lin\'eaire  sur $\mathbb K(z)$ entre les fonctions $f_1(z),\ldots,f_r(z)$ qui se sp\'ecialise au point $1$ en une relation non triviale  entre les nombres $\xi_1,\ldots,\xi_r$. La relation obtenue est n\'ecessairement \`a coefficients dans $\mathbb K$ et on obtient donc que les nombres $\xi_1,\ldots,\xi_r$ sont 
	lin\'eairement d\'ependants sur $\mathbb K$. 
			
	En d'autres termes,  les extensions ${\bf E}_{\mathbb K}$ et $\Q$ sont lin\'eairement disjointes sur $\mathbb K$. Le fait que 
	${\bf E}={\bf E}_\mathbb K\otimes_\mathbb K\Q$ 
	d\'ecoule classiquement de cette propri\'et\'e (voir, par exemple, \cite[Chapter V, §2]{Bo} et \cite[Chapter VII, §3]{La}).  	
\end{proof}

Les deux résultats suivants découlent également du corrolaire \ref{coro: descente}. 
Dans le cas (E), ils correspondent respectivement au Corrolary 3 et au Lemma 1 de 
\cite{FR23}. Dans le cas {\rm(${\rm M}_q$)}, le corollaire \ref{coro: dichotomie} correspond au corollaire 1.8 de \cite{AF17}.

\begin{coro}\label{coro: descente2}
	On a les deux r\'esultats suivants. 
	
	\begin{itemize}
	
	\item[{\rm (E)}]  Soient ${\mathbb K} \subset \Q$ un corps de nombres et $w_1,\ldots,w_d$ une base du $\mathbb Q$-espace vectoriel $\mathbb K$. Alors, on a 
	$$\mbox{\bf E}_{\mathbb K}=w_1\mbox{\bf E}_{\mathbb Q}\oplus\cdots \oplus w_d\mbox{\bf E}_{\mathbb Q}\,.$$ 
	
	\item[{\rm(${\rm M}_q$)}]  Soient ${\mathbb K} \subset \Q$ un corps de nombres,  $\alpha\in \mathbb K$, $0<\vert\alpha\vert<1$, et $w_1,\ldots,w_d$ une base du $\mathbb Q(\alpha)$-espace vectoriel $\mathbb K$. Alors, on a 
	$$\mbox{\bf M}_{q,\alpha,\mathbb K}=w_1\mbox{\bf M}_{q,\alpha,\mathbb Q(\alpha)}\oplus\cdots \oplus w_d\mbox{\bf M}_{q,\alpha,\mathbb Q(\alpha)}\,.$$ 	
	\end{itemize}
\end{coro}

\begin{rem}
On peut montrer que  ${\bf M}_{q,\alpha,\mathbb Q(\alpha)}={\bf M}_{q,\alpha,\mathbb Q}$ et donc simplifier légèrement l'expression du cas {\rm(${\rm M}_q$)}. 
\end{rem}

\begin{coro}\label{coro: dichotomie}
	Soit ${\mathbb K} \subset \Q$ un corps de nombres. On a les deux r\'esultats suivants.
	
	\begin{itemize}
	
	\item[{\rm (E)}]  Si $f(z)$ est une $E$-fonction  à coefficients dans $\mathbb K$ et 
	$\alpha\in \mathbb K$, alors on a l'alternative suivante : soit  $f(\alpha)\in\mathbb K$, soit $f(\alpha)$ est transcendant. 
	
	\item[{\rm(${\rm M}_q$)}]  Si $f(z)$ est une $M_q$-fonction  à coefficients dans $\mathbb K$ et 
	$\alpha\in \mathbb K$, $0<\vert\alpha\vert<1$, n'est pas un p\^ole de $f(z)$, alors on a l'alternative suivante : soit  $f(\alpha)\in \mathbb K$, soit $f(\alpha)$ est transcendant.	
	\end{itemize}
\end{coro}


\section{Quelques remarques liées à un article de Fischler et Rivoal}\label{sec: FR}

Dans la section \ref{sec: descente}, nous avons déduit du théorème \ref{thm: orbital} plusieurs énoncés également obtenus par Fischler et Rivoal pour les $E$-fonctions (à savoir les Theorem 2, Theorem 3 et le Corrolary 3 de \cite{FR23}), ainsi que leurs analogues dans le cadre des $M_q$-fonctions. Ces résultats ont été obtenus indépendamment de \cite{FR23}, contrairement  
aux résultats présentés dans cette section qui ont été inspirés par \cite{FR23}. 

\subsection{Considérations galoisiennes}

Nous démontrons d'abord, pour les $M_q$-fonctions, des résultats analogues aux propositions 1 et 4 de \cite{FR23}. 

\'Etant donnés une série formelle $f(z)=\sum_{n=0}^\infty a(n)z^n\in \Q[[z]]$ et un élément $\tau$ de ${\rm Gal}(\Q/\mathbb Q)$, on pose $f^{\tau}(z):=\sum_{n=0}^\infty \tau(a(n))z^n$. 
On vérifie aisément que si $f$ est une $M_q$-fonction, il en est de même de 
$f^{\tau}$. 
Le résultat suivant est une conséquence directe du  théorème \ref{thm: orbital}. 

\begin{prop}
	\label{prop:conjugates}
	Soient $f$ une $M_q$-fonction et $\alpha \in \Q$, $0<\vert \alpha\vert<1$, tel que $f(\alpha) \in \Q$. Pour tout $\tau \in {\rm Gal}(\Q/\mathbb Q)$ tel que $\vert \tau(\alpha)\vert <1$, 
	on a $f^\tau(\tau(\alpha))=\tau(f(\alpha))$.
\end{prop}

\begin{proof}
	Notons tout d'abord que si $p(z) \in \Q[z]$ on a $p^\tau(\tau(\alpha))=\tau(p(\alpha))$ et que le résultat est évident si $\alpha=0$. On suppose désormais que $\alpha\in\Q^*$. Comme $f(\alpha) \in \Q$, les nombres $1$ et $f(\alpha)$ sont linéairement dépendants sur $\Q$. 
	D'après le théorème \ref{thm: orbital}, une telle relation est la spécialisation en $z=\alpha$ d'une relation $\sigma_q$-linéaire (éventuellement dégénérée) entre les fonctions $1$ et $f(z)$. 
	Cela signifie qu'il existe un entier $m$ et des polynômes $p_{-1}(z),\ldots,p_m(z) \in \Q[z]$ tels que 	
	\begin{equation}
	\label{eq:orbitale_rel}
	p_{-1}(z)+p_0(z)f+p_1(z)\sigma_q(f)+\cdots + p_m(z)\sigma_q^m(f)=0 
	\end{equation}
	et 
	\begin{equation}
	\label{eq:cond_annulation}
	p_{-1}(\alpha) = f(\alpha),\ p_0(\alpha)=-1,\ p_1(\alpha)=\cdots = p_m(\alpha)=0\,.
	\end{equation}
	Comme $\tau$ et $\sigma_q$ commutent, en appliquant $\tau$ à \eqref{eq:orbitale_rel}, on obtient 
		\begin{equation}
	\label{eq:orbitale_rel_conj}
	p_{-1}^\tau(z)+p_0^\tau(z)f^\tau+p_1^\tau(z)\sigma_q(f^\tau)+\cdots + p_m^\tau(z)\sigma_q^m(f^\tau)=0\,.
	\end{equation}
	En spécialisant au point  $\tau(\alpha)$ et en utilisant  \eqref{eq:cond_annulation}, on obtient 
	$
	f^\tau(\tau(\alpha))  = \tau(f(\alpha))$, comme souhaité.	
\end{proof}

\begin{coro}
	Soient $f(z)$ une $M_q$-fonction à coefficients dans un corps de nombres $\mathbb K$ et $\alpha \in \Q$, $\vert \alpha\vert <1$. Les énoncés suivants sont équivalents : 
	\begin{enumerate}[label=(\roman*)]
		\item[{\rm (i)}] $f(\alpha)=0$,
		\item[{\rm (ii)}]   il existe $\tau \in {\rm Gal}(\Q/\mathbb Q)$ tel que  $\vert \tau(\alpha)\vert<1$ et $f^\tau(\tau(\alpha))=0$, 
		\item[{\rm (iii)}]  pour tout $\tau \in {\rm Gal}(\Q/\mathbb Q)$ tel que $\vert \tau(\alpha)\vert<1$, on a $f^\tau(\tau(\alpha))=0$,  
		\item[{\rm (iv)}]  Soit $D$ le polynôme minimal de $\alpha$ sur $\mathbb K$. La $M_q$-fonction 
		$g(z):=f(z)/D(z)\in\mathbb K[[z]]$ est bien définie en $\tau(\alpha)$ pour tout $\tau \in {\rm Gal}(\Q/\mathbb K)$ tel que $\vert \tau(\alpha)\vert <1$. 	
	\end{enumerate}
	\end{coro}

\begin{proof}
	Les implications (iii) $\Rightarrow$ (i) $\Rightarrow$ (ii) sont immédiates.  D'après la proposition \ref{prop:conjugates}, (ii) implique que 
	$\tau(f(\alpha))=0$ et donc $f(\alpha)=0$. On a donc (ii) implique (i). Soit $\tau \in {\rm Gal}(\Q/\mathbb Q)$ tel que $\vert \tau(\alpha)\vert <1$. En combinant, la proposition \ref{prop:conjugates} et 
	(i) on obtient  que 
	$$
	f^\tau(\tau(\alpha)) = \tau(f(\alpha))=\tau(0)=0\,.
	$$
	Ainsi, (i) $\Rightarrow$ (iii).
	Comme (iv) implique trivialement (i),  il suffit par exemple de montrer que (iii) implique (iv). Or, d'après (iii), pour tout $\tau \in {\rm Gal}(\Q/\mathbb K)$, tel que $\vert \tau(\alpha)\vert <1$, on a $f^\tau(\tau(\alpha))=0$. Comme $f^\tau=f$, on en déduit que $f$ s'annule en $\tau(\alpha)$ et donc que la propriété (iv) est vérifiée.   	
\end{proof}

Similairement à \cite[Section 6]{FR23}, on définit  une action de ${\rm Gal}(\Q/\mathbb Q(\alpha))$  sur $ {\bf M}_{q,\alpha}$ de la façon suivante : étant donnés $\xi \in {\bf M}_{q,\alpha}$ et $\tau \in{\rm Gal}(\Q/\mathbb Q(\alpha))$ on pose $\tau(\xi) = f^\tau(\alpha)$ où $f$ désigne une $M_q$-fonction telle que $f(\alpha)=\xi$. Pour montrer que cette action est bien définie, il suffit de vérifier que 
$\tau(\xi)$ ne dépend pas du choix de $f$. Soit $g$ une $M_q$-fonction telle que $g(\alpha)=\xi$ et $h:=f-g$. On a $h(\alpha)=0$. Il découle de la proposition  \ref{prop:conjugates} 
que $h^\tau(\tau(\alpha))=h^\tau(\alpha)=0$. Comme $h^\tau=f^\tau -g^\tau$, on a $f^\tau(\alpha)=g^\tau(\alpha)$, comme souhaité. 

Le résultat suivant est l'analogue de \cite[Proposition 4]{FR23}. 

\begin{prop}
Soit $\mathbb K$ un corps de nombre, $\alpha \in \mathbb K$ et $\xi \in {\bf M}_{q,\alpha}$. Alors, $\xi \in {\bf M}_{q,\alpha,\mathbb K}$ si et seulement si $\tau(\xi)=\xi$ pour tout $\tau\in {\rm Gal}(\Q/\mathbb K)$.
\end{prop}

\begin{proof}
	Soit $\xi \in {\bf M}_{q,\alpha,\mathbb K}$ et $f\in \mathbb K[[z]]$ une $M_q$-fonction telle que $\xi=f(\alpha)$. Alors $\tau(\xi)=f^\tau(\alpha)=f(\alpha)=\xi$ pour tout $\tau \in {\rm Gal}(\Q/\mathbb K)$.
	
	Supposons à présent que $\xi$ est fixé par tous les éléments de ${\rm Gal}(\Q/\mathbb K)$. Soit $f$ une $M_q$-fonction telle que $f(\alpha)=\xi$. Soit $\mathbb L$ une extension galoisienne de $\mathbb K$ contenant tous les coefficients de $f$. En définissant $g$ comme la moyenne des $f^\tau(z)$ quand $\tau$ parcourt ${\rm Gal}(\mathbb L/\mathbb K)$, on obtient que 
	$g(z) \in \mathbb K[[z]]$ et  $g(\alpha)=\xi$, puisque  $\tau(\xi)=\xi$ pour tout $\tau \in{\rm Gal}(\mathbb L/\mathbb K)$.  Ainsi, $\xi \in {\rm M}_{q,\alpha,\mathbb K}$.
\end{proof}

\subsection{Décomposition des $M_q$-fonctions sur un corps de nombres}

Soit $R$, $0<R\leq 1$, un nombre réel. Une $M_q$-fonction $f$ est dite purement transcendante sur $D(0,R)$ si $f$ prend des valeurs transcendantes en tout point algébrique $\alpha$ du disque épointé $D^*(0,R)$ qui n'est pas un pôle de $f$. 

Le résultat suivant est l'analogue de \cite[Proposition 3]{FR23}. 

\begin{prop}\label{prop:decompo}
	Soient $\mathbb K$ un corps de nombre, $f(z) \in \mathbb K[[z]]$ une $M_q$-fonction et $R$, $0<R< 1$, un nombre réel. 
	Alors, il existe des fractions rationnelles $R_1(z),R_2(z) \in \mathbb K(z)$ et une $M_q$-fonction $g(z) \in \mathbb K[[z]]$, analytique et purement transcendante sur $D(0,R)$, tels que
	$$
	f(z) = R_1(z)+R_2(z)g(z)\,.
	$$
\end{prop}


Soient $f(z)$ une $M_q$-fonction et $\alpha\in\Q$, $0<\vert \alpha\vert<1$, qui n'est pas un pôle de $f$. La \textit{multiplicité} de $f$ en $\alpha$ 
est  définie comme le supremum des entiers $m$ tel qu'il existe  $P(z) \in \Q[z]$ de degré au plus $m-1$ pour lequel le quotient
$$
\frac{f(z) - P(z)}{(z-\alpha)^m}
$$ 
est bien défini au point $\alpha$. Si 
$
f(z)=\sum_{n} \eta_n (z-\alpha)^n$ 
est le développement de Taylor de $f(z)$ en $\alpha$, 
alors $m$ est le plus petit entier pour lequel $\eta_m \notin \Q$. Remarquons que la multiplicité est nulle si et seulement si $f(\alpha) \notin \Q$. 
D'autre part, le lemme suivant assure que la multiplicité de $f$ en $\alpha$ est toujours finie.

\begin{lem}\label{lem:multiplicité}
Soient $1,f(z),f_2(z),\ldots,f_n(z)$ des $M_q$-fonctions linéairement indépendantes  sur $\Q(z)$ et formant un vecteur solution d'un système $q$-mahlé\-rien linéaire associé à une matrice  $A(z)$. 
Supposons que ces fonctions sont définies en $\alpha$ et que 
$\alpha^q$ est un point régulier pour ce système. Alors,  la multiplicité de $f$ en $\alpha$  est majorée par l'ordre de $\alpha$ comme pôle de $\det A(z)$. 
\end{lem}

\begin{proof}
On procède par récurrence sur l'ordre $\nu$ de $\alpha$ en tant que pôle de $\det A(z)$.

Supposons que $\nu=0$, c'est-à-dire que $A(\alpha)$ est bien définie. Comme le point $\alpha^{q}$ est régulier pour ce système et que les fonctions sont toutes définies en $\alpha$, 
la matrice $A^{-1}(\alpha)$ est également définie. Le fait que les matrices $A(z)$ et $A^{-1}(z)$ soient bien définies en $\alpha$ et que $\alpha^q$ soit un point régulier implique que le point $\alpha$ est lui-même régulier.  Puisque les fonctions $1,f(z),f_2(z),\ldots,f_n(z)$ sont  linéairement indépendantes, il résulte du théorème M2 que $f(\alpha) \notin \Q$. Ainsi, la multiplicité de $f$ en 
$\alpha$ est  nulle et donc bien inférieure ou égale à $\nu$. 

Soit $\nu \geq 1$. Supposons le résultat montré pour l'ordre $\nu-1$. Supposons maintenant que l'ordre de $\alpha$ comme pôle de $\det A(z)$ est égal à $\nu$. 
Si $f(\alpha) \notin \Q$, sa multiplicité est nulle, donc inférieure à $\nu$. On peut donc supposer que $f(\alpha)  \in \Q$. 
Posons $g(z) := (f(z)-f(\alpha))/(z-\alpha)$ et notons que $g$ est une $M_q$-fonction qui est bien définie en $\alpha$. 
On obtient que 
$$
\begin{pmatrix}
1 \\ g(z^q) \\ f_2(z^q) \\ \vdots \\ f_n(z^q) 
\end{pmatrix}
=
P(z^q)A(z)P(z)^{-1}
\begin{pmatrix}
1 \\ g(z) \\ f_2(z) \\ \vdots \\ f_n(z) 
\end{pmatrix}
$$
où 
$$
P(z):= 
\begin{pmatrix}
1 & 0 & 0 & \cdots & 0 
\\
\frac{-f(\alpha)}{z-\alpha} & \frac{1}{z-\alpha}& 0& \cdots & 0 
\\ & & 1 & &
\\ &&& \ddots 
\\ &&&&1
\end{pmatrix}\,.
$$
Posons $B(z):=P(z^q)A(z)P(z)^{-1}$. Comme 
$$
\det B(z)=\frac{\det A(z)(z - \alpha)}{z^q - \alpha} \,,
$$
l'ordre de $\alpha$ comme pôle de $\det B(z)$ est égal à $\nu - 1$. Par hypothèse de récurrence, on obtient que la multiplicité de $g$ en $\alpha$  est au plus $\nu - 1$ et 
donc que la multiplicité de $f$ en $\alpha$  est au plus $\nu$. 
\end{proof}

Démontrons d'abord la Proposition \ref{prop:decompo} dans le cas où $f$ est analytique sur le disque $D(0,R)$. 

\begin{prop}\label{prop: decompo2}
	La proposition \ref{prop:decompo} est vraie si $f$ est analytique sur $D(0,R)$. 
	\end{prop}

\begin{proof}
Nous procédons par récurrence sur la somme $S$ des multiplicités de $f$ aux points de $\Q\cap D(0,R)$. Notons que $S$ est  finie d'après le lemme \ref{lem:multiplicité} 
et le fait que la multiplicité de $f$ est non nulle seulement en un nombre fini de point de $\Q\cap D(0,R)$.  

Supposons que $S=0$. Alors, $f(z)$ prend des valeurs transcendantes en tout point de $D^*(0,R)$ et il suffit de choisir $P:=0$, $Q:=1$ et $g:=f$. 

Supposons à présent que $S \geq 1$ et que le résultat est démontré pour toute fonction analytique sur le disque $D(0,R)$ dont la somme des multiplicités est au plus $S-1$. 
Soit $\alpha$ un point de multiplicité $m\geq 1$ pour $f$. Soit $\mathbb L$ une extension galoisienne de $\mathbb K$ contenant $\alpha$ et posons 
$G:={\rm Gal}(\mathbb L/\mathbb K)$. 
Notons que d'après le corollaire \ref{coro: dichotomie}, on a $f(\alpha) \in \mathbb K$. Posons
$$
h(z) :=  \sum_{\tau \in G} \frac{f(z) - \tau(f(\alpha))}{z-\tau(\alpha)}\,\cdot
$$
C'est une $M_q$-fonction. De plus, comme $h(z)$ est fixée par $G$, on a $h(z)\in\mathbb K[[z]]$. 
Montrons que $h$ est analytique sur $D(0,R)$. Il est immédiat que $h$ est analytique en tout point  de $D(0,R)$ qui n'est pas de la forme $\tau(\alpha)$ pour un $\tau \in G$. 
D'autre part, si $\tau(\alpha)\in D(0,R)$, $\tau \in G$,  la proposition \ref{prop:conjugates} implique que $\tau(f(\alpha))=f(\tau(\alpha))$ et donc que $h$ 
est analytique au point $\tau(\alpha)$. Ainsi,  $h$ est analytique sur $D(0,R)$. 
La multiplicité en $\alpha$ de la fonction
$$
\frac{f(z) - \tau(f(\alpha))}{z-\tau(\alpha)}
$$
vaut $m$ si $\tau(\alpha)\neq \alpha$ et $m-1$ sinon. On en déduit que la multiplicité de $h$ en $\alpha$  vaut $m-1$. En tout autre point de $\Q\cap D(0,R)$, 
la multiplicité de $h$ est inférieur ou égale à celle de $f$. La somme des multiplicités de $h$ est donc au plus égale à $S-1$. 
Par hypothèse de récurrence, il existe  $A_1(z),A_2(z) \in \mathbb K(z)$ et  $g(z)\in \mathbb K[[z]]$ analytique et purement transcendante sur $D(0,R)$, tels que
$$
h(z)=A_1(z)+A_2(z)g(z)\,.
$$
On définit alors
$$
\delta(z) := \left( \sum_{\tau \in G} \frac{1}{z- \tau(\alpha)}\right)^{-1} \text{ et } \gamma(z) := \delta(z) \left(\sum_{\tau\in G} \frac{\tau(f(\alpha))}{z-\tau(\alpha)}\right)
$$
Par construction, $\delta(z)$ et $\gamma(z)$ appartiennent à $\mathbb K(z)$. Par ailleurs, on a
$$
f(z)=\delta(z)h(z)-\gamma(z)\,.
$$
En posant 
$R_1(z):=A_1(z)\delta(z)-\gamma(z)$ et $R_2(z):=A_2(z)\delta(z)$, 
on obtient  
$$
f(z)=R_1(z) + R_2(z)g(z)
$$
comme souhaité.
\end{proof}

\begin{proof}[Démonstration de la Proposition \ref{prop:decompo}]
Nous procédons par récurrence sur la somme $S$ des pôles de $f$ dans $D(0,R)$, comptés avec multiplicité. 

Si $S=0$, $f$ est analytique sur $D(0,R)$ est  le résultat correspond à la proposition \ref{prop: decompo2}. 

Supposons à présent que $S \geq 1$ et que le résultat est démontré quand la somme des pôles est au plus $S-1$. 
Soit $\alpha$ un pôle de $f(z)$ et $D(z)$ le polynôme minimal de $\alpha$ sur $\mathbb K$. Alors la somme des pôles de $D(z)f(z)\in\mathbb K[[z]]$ est strictement inférieure à $S$. 
Par hypothèse de récurrence, il existe $A_1(z),A_2(z) \in \mathbb K(z)$ et  $g(z)\in\mathbb K[[z]]$, analytique et purement transcendante sur $D(0,R)$, 
tels que
$$
D(z)f(z)=A_1(z)+A_2(z)g(z) \,.
$$ 
En posant $R_1(z):=A_1(z)/D(z)$ et $R_2(z):=A_2(z)/D(z)$, on obtient le résultat souhaité. 
\end{proof}
\section{Exemples de relations non banales}\label{sec: ex}

Nous illustrons le théorème \ref{thm: orbital} à travers quelques exemples. 

\begin{ex}
D'après la proposition \ref{prop: sporadique}, le lieu des relations non banales entre des $E$-fonctions fixées est un ensemble fini. Réciproquement, tout ensemble fini $\mathcal S\subset \Q^*$ est le lieu des relations banales entre certaines $E$-fonctions. En effet, considérons $f(z):=P(z)e^z$, où $P(z)\in\Q[z]$ est à racines simples avec $\mathcal S$ comme ensemble de racines. On obtient que $f(\alpha)=0$ pour tout $\alpha\in\mathcal S$ et $f(\alpha)\not\in\Q$ pour tout $\alpha\in\Q^*\setminus\mathcal S$ . Ainsi, l'ensemble $\mathcal S$ correspond au lieu des relation non banales pour $\{f\}$.  Notons que pour tout $\alpha\in\mathcal S$, la relation $f(\alpha)=0$ s'obtient comme dégénérescence de la relation différentielle 
$$
(P'+P)f -Pf'=0\,.
$$
\end{ex}

\begin{ex}
	Considérons la fonction de Bessel 
	$$
	J_0(z)=\sum_{n=0}^\infty \frac{(-1)^n}{(n!)^2}\left(\frac{x}{2}\right)^{2n}
	$$
	et 
	$$
	f(z):=\sum_{n=0}^\infty\left(\frac{(-1)^n(2n)!}{(n!)^42^{2n}}-\frac{(-1)^n}{n!(n-1)!2^{2n-1}}\right)x^{2n}-\sum_{n=0}^\infty \frac{(-1)^{n}}{(n+1)!n!}\frac{x^{2n+1}}{2^{2n+1}} \,\cdot
	$$
	En $z=1$, on a la relation algébrique suivante : 
	\begin{equation}\label{eq:  relbessel}
	f(1)-J_0(1)^2=0\,.
	\end{equation}
	On peut vérifier que cette relation s'obtient par dégénérescence de la relation algébro-différentielle suivante entre $f$ et $J_0$ :
	\begin{equation}\label{eq: jf}
	f(z)-J_0(z)^2+(z-1)J_0'(z)=0\,. 
	\end{equation}
	Comme $J_0(z)$ et $J_0'(z)$ sont algébriquement indépendantes sur $\Q(z)$, il en va de même, d'après \eqref{eq: jf}, de $f(z)$ et $J_0(z)$. En particulier, la  relation \eqref{eq: relbessel} est non banale  relativement 
	à $\{f(z), J_0(z)\}$.
\end{ex}

\begin{ex}\label{ex: tm3}
Considérons $f(z):=\sum_{n=0}^\infty a_nz^n$, où $a_n\in\{0,1\}$ désigne le nombre d'occurences du chiffre $2$ dans l'écriture ternaire de $n$, compté modulo $2$.  
La suite $(a_n)_{n\geq 0}$ est une variante de la célèbre suite de Thue-Morse.  
Des relations $a_{3n}=a_{3n+1}=a_n$ et $a_{3n+2}=1-a_n$, on déduit que $f(z)$ est une $M_3$-fonction, solution de l'équation $3$-mahlérienne inhomogène 
\begin{equation}\label{eq: tm3}
\frac{z^2}{1-z^3} - f(z) + (1+z-z^2)f(z^3) =0 \,.
\end{equation}
Posons $\varphi := \frac{1-\sqrt{5}}{2}$. La relation $\sigma_3$-linéaire \eqref{eq: tm3} entre $1$ et $f$ dégénère au point $\varphi$ en 
$$
\frac{\varphi^2}{1-\varphi^3}-f(\varphi)=0\,.
$$
En particulier, $f(\varphi) \in \Q$. En raisonnant par récurrence à partir de \eqref{eq: tm3}, on obtient que pour tout $\ell\geq 0$, $f(\varphi^{1/3^\ell}) \in \Q$. 
Par ailleurs, on peut montrer que $f$ n'est pas rationnelle et l'on obtient donc qu'il existe une relation linéaire non banale relativement à $\{1,f\}$ 
en tout point de la forme $\varphi^{1/3^\ell}$, $\ell\geq 0$.  
Le lieu des relations non banales est donc infini, bien qu'il soit fini sur tout disque $D(0,R)$, $0<R<1$, comme le garantit la proposition \ref{prop: sporadique}.  

Notons qu'une relation linéaire entre $1$ et $f(\varphi^{1/3})$ s'obtient par dégénérescence de la relation $\sigma_3$-linéaire 
\begin{equation*}
\frac{z^2+z^5+z^6+z^7}{1-z^9} - f(z) + (1+z-z^2)(1+z^3-z^6)f(z^9) =0\,,
\end{equation*}
laquelle s'obtient à partir de \eqref{eq: tm3} et donne 
$$
 \frac{\varphi^{2/3}+\varphi^{5/3}+\varphi^{6/3}+\varphi^{7/3}}{1-\varphi^3}-f(\varphi^{1/3})=0 \,.
$$
De façon générale, on peut montrer que toute relation linéaire entre $1$ et $f(\varphi^{1/3^{\ell}})$ peut s'obtenir par dégénérescence d'une relation $\sigma_3$-linéaire 
de profondeur $\ell+1$, c'est-\`a-dire faisant intervenir la fonction $\sigma_3^{\ell+1}(f)$ et qu'on ne peut l'obtenir comme dégénérescence d'une relation de profondeur moindre. 

Cet exemple met en évidence une différence avec  le cas des $E$-fonctions, où l'on peut obtenir toutes les relations non banales entre  $f_1(\alpha),\ldots,f_r(\alpha)$ par dégénérescence de relations algébro-différentielles dont la profondeur est bornée indépendamment de $\alpha$. 
Pour obtenir un résultat analogue pour les $M_q$-fonctions,  il faut alors  restreindre  leur étude au disque $D(0,R)$ où $R<1$ est préalablement fixé. 
\end{ex}

\begin{ex}\label{ex: bs}
Les suites de Baum-Sweet et Rudin-Shapiro figurent parmi les exemples classiques de suites automatiques (cf.  \cite[Chapter 5]{AS03}). 
	Notons $f(z)$ la série génératrice associée à la suite de Baum-Sweet et $g(z)$ celle associée à la suite de Rudin-Shapiro. 
	Ce sont deux $M_2$-fonctions solutions des équations suivantes :
	$$
	f(z)-zf(z^2) - f(z^4)=0\; \;\;\text{ et }\;\;\; g(z) +(z-1)g(z^2)-2zg(z^4)=0\, .
	$$
	Considérons la $M_2$-fonction $h(z):=(1-3z)f(z^2)^3 + f(z)g(z)$. D'après \cite[Section 9.3]{Ro18}, les fonctions $f(z)$, $f(z^2)$, $g(z)$ et $g(z^2)$ sont algébriquement indépendantes sur 
	$\Q(z)$. On en déduit que  $h(z)$, $f(z)$ et $g(z)$ le sont également.  Par ailleurs, on a 
	\begin{equation}\label{eq:P2-ex_relation_alg_bs_rs}
	f\left(\frac{1}{3}\right)g\left(\frac{1}{3}\right) -h\left(\frac{1}{3}\right)= 0\, ,
	\end{equation}
	qui est donc nécessairement une relation non banale relativement à $\{f,g,h\}$. Elle s'obtient par dégénérescence en $z=1/3$ de la relation $\sigma_2$-algébrique  
	$$
	f(z)g(z) -h(z) +(1-3z)f(z^2)^3 =0\,.
	$$
\end{ex}

\begin{ex}
Il peut également exister des relations $\sigma_q$-algébriques non triviales, mais qui ne sont source d'aucune dégénérescence. Voici un exemple. 
	On note à nouveau $f(z)$ la série génératrice de la suite de Baum-Sweet (voir  exemple \ref{ex: bs}). Notons $(n)_2$ le développement binaire de l'entier $n$ et $S_2(n)$ 
	la somme des chiffres  binaires de $n$. Définissons alors la suite $(b_n)_{n \in \N}$ par :
	$$
\left\{	\begin{array}{rcll}
b_n &=& 0 & \text{si } (n)_2 \text{ a un bloc de $0$ consécutifs de longueur impair},
\\
b_n &=& 1 & \text{si $(n)_2$ ne contient pas de tel bloc et $S_2(n)$ est pair},
\\ b_n&=&-1 & \text{sinon}.\end{array}\right.
	$$
	On pose $g(z):=\sum_{n\in \N} b_nz^n$. On peut vérifier que 
	$$
	\begin{pmatrix}
	f(z) & g(z)\\
	f(z^2) & -g(z^2)
	\end{pmatrix}
	=\begin{pmatrix}
	z & 1 \\ 1 & 0
	\end{pmatrix}
		\begin{pmatrix}
	f(z^2) & -g(z^2)\\
	f(z^4) & g(z^4)
	\end{pmatrix}\,.
	$$
	D'après \cite{Ro18}, le groupe de Galois de ce système $2$-mahlérien est ${\rm SL}_2(\Q)$.
	On peut alors utiliser la proposition \ref{prop: sigmaqideal} pour montrer que les relations $\sigma_2$-algébriques entre $f$ et $g$ sont engendrées par la relation 
	\begin{equation}\label{eq:rel_BaumSweet}
	f(z)g(z^2)+f(z^2)g(z)=2\,
	\end{equation}
	et leurs équations minimales respectives : 
	$$
	f(z)-zf(z^2)-f(z^4) = 0\qquad \mbox{ et } \qquad g(z) +zg(z^2) - g(z^4)=0\,.
	$$ 
	 Cependant, ces relations ne sont source d'aucune dégénérescence :  $f(\alpha)$ et $g(\alpha)$ sont algébriquement indépendants sur $\Q$ pour tout $\alpha\in\Q$, $0<\vert\alpha\vert<1$.
\end{ex}


\end{document}